\documentclass[12pt]{article}

\usepackage{amsmath}
\usepackage{amssymb}
\usepackage{amsthm}
\usepackage{amsfonts}
\usepackage{amscd}
\usepackage[numbers,sort&compress]{natbib}
\usepackage{hyperref}
\usepackage[nottoc]{tocbibind}  
\usepackage[usenames,dvipsnames]{color}
\usepackage{enumitem} 
\usepackage{titling}

\hypersetup{
    colorlinks=true,
    breaklinks=true,
    filecolor=blue,
    urlcolor=Violet,
    linkcolor=blue,
    citecolor=Green,
    anchorcolor=blue,
    frenchlinks=true,
    pdfborder={0 0 0},
    pdfdisplaydoctitle=true,
    bookmarksnumbered=true
}

\theoremstyle{plain}
\newtheorem{lemma}{Lemma}
\newtheorem{proposition}{Proposition}
\newtheorem{theorem}{Theorem}
\newtheorem{corollary}{Corollary}

\theoremstyle{definition}
\newtheorem{definition}{Definition}
\newtheorem{remark}{Remark}
\newtheorem{example}{Example}
\newtheorem{algorithm}{Algorithm}
\newtheorem{assumption}{Assumption}

    \newcommand{\Mend}{\hfill \ensuremath{\vartriangleleft}}

    \newcommand{\Mrnk}{\ensuremath{\textrm{rank}}}

    \newcommand{\Msum}[2]{\ensuremath{\overset{#2}{\underset{#1}{\sum}}}}

    \newcommand{\Mmod}[1]{\langle #1\rangle} 

    \newcommand{\Mset}[2]{\ensuremath{\{~ #1 ~|~ #2 ~\}}}
    \newcommand{\Mfun}[5]{\ensuremath{#1\colon #2 \rightarrow #3,\quad #4 \mapsto #5 }}

    \newcommand{\Miff}{if and only if }

    \newcommand{\Mst}{{such that }}
    \newcommand{\Mwlog}{without loss of generality }
    \newcommand{\Mresp}{respectively}
    \newcommand{\Mwrt}{with respect to }

    \newcommand{\Marrow}[3]{\ensuremath{#1\stackrel{#2}{\longrightarrow}#3}}
    
    \newcommand{\Mdasharrow}[3]{\ensuremath{#1\stackrel{#2}{\dashrightarrow}#3}}
    
    \newcommand{\Mrow}[3]{\ensuremath{#1\colon #2\longrightarrow #3}}
    
    \newcommand{\Mdashrow}[3]{\ensuremath{#1\colon #2\dashrightarrow #3}}
    \newcommand{\Mhookrow}[3]{\ensuremath{#1\colon #2\hookrightarrow #3}}
    
    \newcommand{\Mdef}[1]{\textit{#1}\index{#1}}
    \newcommand{\MdefAttr}[2]{\textit{#1}\index{#2!#1}\index{#1}}

    \newcommand{\Me}{e}
    \newcommand{\Mp}{\varepsilon}
    \newcommand{\Mk}{k}
    \newcommand{\Ml}{\ell}
    \newcommand{\Mh}{h}
    
    \newcommand{\Mkk}{\hat{k}}
    \newcommand{\Mhh}{\hat{h}}
    \newcommand{\Mi}{\mathfrak{i}}

    \newcommand{\McalE}{{\mathcal{E}}}\newcommand{\McalH}{{\mathcal{H}}}\newcommand{\McalP}{{\mathcal{P}}}
    \newcommand{\MbbC}{{\mathbb{C}}}\newcommand{\MbbP}{{\mathbb{P}}}\newcommand{\MbbR}{{\mathbb{R}}}\newcommand{\MbbZ}{{\mathbb{Z}}}
    
    \newcommand{\Mtlf}{{\tilde{f}}}

    \newcommand{\MmfF}{{\mathfrak{F}}}\newcommand{\MmfM}{{\mathfrak{M}}}

    \newcommand{\LEM}[1]{Lemma~\ref{lem:#1}}
    \newcommand{\PRP}[1]{Proposition~\ref{prp:#1}}
    \newcommand{\COR}[1]{{Corollary~\ref{cor:#1}}}
    \newcommand{\THM}[1]{{Theorem~\ref{thm:#1}}}
    \newcommand{\DEF}[1]{{Definition~\ref{def:#1}}}
    \newcommand{\RMK}[1]{{Remark~\ref{rmk:#1}}}
    \newcommand{\EQN}[1]{{(\ref{eqn:#1})}}
    \newcommand{\SEC}[1]{{\textsection\ref{sec:#1}}}
    
    \newcommand{\ALG}[1]{{Algorithm~\ref{alg:#1}}}
    
    \newcommand{\EXM}[1]{{Example~\ref{exm:#1}}}

    \newcommand{\ASM}[1]{{Assumption~\ref{asm:#1}}}

    \newcommand{\Mmclaim}[1]{\item[\textbf{#1)}]}
    \newcommand{\Mrefmclaim}[1]{#1)}
    \newcommand{\Mclaim}[1]{\textit{Claim #1:}}
    \newcommand{\Mrefclaim}[1]{claim #1}

\title{Minimal degree rational curves on real surfaces}

\author{Niels Lubbes}

\date{\today}

\begin{document}

\maketitle

\begin{abstract}
We classify real families of minimal degree rational curves that cover 
an embedded rational surface.
A corollary is that
if the projective closure of a smooth surface is not biregular isomorphic to the projective closure of the unit-sphere, 
then the set of minimal degree rational curves that cover
the surface is either empty or of dimension at most two.
Moreover, if these curves are of minimal degree over the real numbers,
but not over the complex numbers, then almost all the curves are smooth.
Our methods lead to an algorithm that takes as input a 
real surface parametrization and outputs all real 
families 
of rational curves of lowest possible degree
that cover the image surface.
\\[2mm]
\\[2mm]
\end{abstract}

\begingroup
\def\addvspace#1{\vspace{-1mm}}
\tableofcontents
\endgroup

\section{Introduction}

Lines play a central role in classical geometry and 
have been further developed as geodesics in Riemannian geometry.
From an algebro geometric point of view, we can consider lines as rational curves of minimal degree.
Rational curves of low degree play an important role in the complex classification 
of higher dimensional varieties \citep[Chapter V]{kol1}, \citep[page 342]{hwa1}.

We present in \THM{cls} a classification of 
real families of minimal degree rational curves that cover an 
$\MbbR$-rational surface in projective space.
In \SEC{fam} we define such families as \Mdef{minimal families}. 
See \SEC{rat} for the definition of $\MbbR$-rationality. 
The hypothesis for \THM{cls} requires actually less than $\MbbR$-rationality, but needs terminology
from \SEC{chain}.
We conjecture that the $\MbbR$-rationality assumptions in \THM{cls} and \COR{cls} can be omitted.

The plane is covered by a 2-dimensional family of lines, and --- over the 
complex numbers --- any smooth projective surface that is covered by a minimal family of dimension at least two,
must be biregular isomorphic to the projective plane.
It follows from our classification that --- over the real numbers ---
some surfaces are covered by both 1-dimensional and 2-dimensional minimal families (see \EXM{chain}).

In a time span of two thousand years, geometers have dedicated their lives  
to prove the parallel postulate using Euclid's first four postulates,
before finally the projective closure of the unit-sphere turned out to be a natural space for non-Euclidean geometry.
Indeed, the 2-sphere is covered by a 3-dimensional family of circles.
When are curves on a real surface like circles on the sphere?
\COR{cls}.\Mrefmclaim{a} provides a characterization.
 
\begin{corollary}
\textbf{(classification of real minimal families)}
\label{cor:cls}
\begin{itemize}[topsep=0pt, itemsep=0pt]
\Mmclaim{a}
Suppose that $X\subset\MbbP^n$ is the projective closure of an 
$\MbbR$-rational surface $S\subset\MbbR^n$ that is covered by a
family of real rational curves that are of minimal degree.
If this family is of dimension at least three, 
then the linear normalization of $X$
is biregular isomorphic to the projective closure of the unit-sphere.

\Mmclaim{b}
The canonical degree of a minimal family on a $\MbbR$-rational embedded surface 
is either -2, -3 or -4, such that the family is of dimension 1, 2 and 3 \Mresp.

\Mmclaim{c}
A minimal family on a $\MbbR$-rational embedded surface 
that is not minimal over the complex numbers,
must be complete. In other words, a minimal family of curves
that are singular outside the singular locus of the surface
must also be minimal over the complex numbers.
\end{itemize}
\end{corollary}

We make the notions of canonical degree and completeness of minimal families precise in \SEC{fam}.
See \SEC{chain} for the definition of linear normalization.
For \COR{cls}.\Mrefmclaim{a} of \THM{cls}, we mention that if $X$ is smooth, then it is biregular isomorphic to its 
linear normalization. 
Moreover, by \citep[VI.6.5]{sil1}, we can replace the $\MbbR$-rationality hypothesis with 
the condition that the real points of $X$ form a connected set.
A remark for \COR{cls}.\Mrefmclaim{b} is that the canonical degree of \emph{complex} minimal families is either -2 or -3.
The canonical degree of a complex minimal rational curve in a family on a smooth $n$-dimensional Fano variety 
is greater than $-(n+1)$ \citep[Theorem~V.1.6]{kol1}.   

Minimal families that cover a given surface, give insight into the 
geometry of this surface.
For example, it was discovered by \citep[Christopher Wren, 1669]{wrn1}
that a one-sheeted hyperboloid contains two lines through each point.
Sir Wren used his discovery for an {\it``engine designed for grinding hyperbolic lenses''} \citep[page~92]{bur1}.
It was shown by astronomer Yvon Villarceau,
that the ring torus is covered by a minimal family of {\it Villarceau circles} \citep[1848]{vil1}. 
These classical discoveries already indicate the interest of minimal degree rational curves in 
geometric modeling. 
Surfaces that are covered by lines or conics are of recent interest in architecture \cite{pot2}.
This article is therefore meant interdisciplinary and our methods are constructive.
We present \ALG{fam} for computing minimal families that cover a given real rational surface.

The idea for the classification in \THM{cls} is
to construct a birational map  $\Mdasharrow{X}{}{Y}$, where $Y$ 
is either a geometrically ruled surface or a weak del Pezzo surface. 
The method of adjunction in \SEC{chain} ensures that this map is (almost) unique up to biregular isomorphism. 
The generators of the Neron-Severi lattice of $Y$,
together with the classes of the pullback of exceptional curves that are contracted
by the map, generate the Neron-Severi lattice of $X$ in a unique way.
In \THM{cls} we classify the divisor classes of minimal degree rational curves \Mwrt these generators.
Notice that the intersection numbers between families
is a topological property of the real surface and can be recovered from this classification.

We classified complex minimal families in \citep[Theorem~46]{nls1} and \citep[Theorem~10]{nls-f3}.
This paper is mostly self-contained and we recover in \RMK{complex} the complex classification as well.
In \COR{conic} we classify families of conics that cover a real surface.
\COR{conic} extends the classification of multiple conical surfaces in \citep[Theorem~8 and Theorem~10]{sch6}.

\section{Preliminaries}

\subsection{}
\label{sec:rat}
A real variety $X$ is defined as a complex variety together with 
an antiholomorphic involution $\Mrow{\sigma}{X}{X}$, which represents the \Mdef{real structure}. 
We implicitly assume that all structures are compatible with $\sigma$ unless explicitly stated otherwise.
For example, if $\Mdasharrow{\MbbP^2}{}{X}$ is a birational map, then
$X$ is rational over the real numbers. We say in this case that $X$ is \Mdef{$\MbbR$-rational}.

\subsection{}
\label{sec:ns}
Recall that the algebraic-, numerical- and linear-equivalence
relations on divisor classes are the same on rational surfaces.
Due to the constructive nature of this paper we make the data associated to the
Neron-Severi lattice explicit. 
The \Mdef{Neron-Severi lattice} $N(X)$ (or \Mdef{NS-lattice} for short) of a rational surface~$X\subset \MbbP^n$ 
consists of the following data:
\begin{enumerate}
\item
A unimodular lattice defined by divisor classes on its smooth model~$Y$ modulo numerical equivalence.
Recall that a \Mdef{smooth model} of a singular surface~$X$ is a birational morphism $\Marrow{Y}{}{X}$ 
from a nonsingular surface~$Y$, that does not contract exceptional curves.

\item
A basis for the lattice.
We shall consider two different bases for $N(X)$:
\begin{itemize}[topsep=0pt]
\item \Mdef{type 1}: $\Mmod{\Me_0,\Me_1,\ldots,\Me_r}$ where the nonzero intersections are $\Me_0^2=1$ and $\Me_j^2=-1$ for $0<j\leq r$,
\item \Mdef{type 2}: $\Mmod{\Ml_0,\Ml_2, \Mp_1,\ldots,\Mp_r}$ 
\Mst the nonzero intersections are $\Ml_0\cdot \Ml_1=1$ and $\Mp_j^2=-1$ for $0<j\leq r$.
\end{itemize}

\item
A unimodular involution $\Mrow{\sigma_*}{N(X)}{N(X)}$ induced by the real structure of~$X$.

\item
A function $\Mrow{h^0}{N(X)}{\MbbZ_{\geq0}}$ assigning the dimension of global sections
of the line bundle associated to a class.

\item
Two distinguished elements $\Mh,\Mk\in N(X)$ corresponding to
class of a hyperplane sections and the canonical class respectively.
\end{enumerate}

\subsection{}
\label{sec:fam}

Suppose that $X\subset \MbbP^n$ is a surface, 
$B$ a smooth variety and $F\subset X\times B$ a divisor.
We call $F$ a \Mdef{family of curves} of $X$, or \Mdef{family} for short,
if the second projection $\Mrow{\pi_2}{F}{B}$ is dominant. 
A \Mdef{member} of $F$, 
corresponding to $b\in B$, is defined as the curve $F_b:=(\pi_1\circ\pi_2^{-1})(b)\subset X$.
We can associate to a curve $C\subset X$ its class $[C]\in N(X)$ \Mst
classes $[C]$ and $[C']$ are equal \Miff $C$ and $C'$ are members of some family.
The class $[F]$ of $F$ is defined as the class of any of its members.
\begin{itemize}[topsep=0pt]
\item
We call $F$ \MdefAttr{covering}{family} if the first projection $\Mrow{\pi_1}{F}{X}$ is dominant. 

\item
We call $F$ \MdefAttr{rational}{family} if the general member of $F$ has geometric genus $0$.

\item
The \MdefAttr{dimension}{family} of $F$ is defined as $\dim B$. Thus a 0-dimensional family consists of a single curve.
If $F$ is complete, then $\dim F = h^0([F])-1$.

%

\item
The \Mdef{degree} of $F$ is defined as the degree of any member \Mwrt the embedding $X\subset\MbbP^n$.
Equivalently, the degree of $F$ is $\Mh\cdot [F]$, where $\Mh\in N(X)$ denotes the class of hyperplane sections. 

\item
The \Mdef{canonical degree} of $F$ is defined as $\Mk\cdot [F]$, where $\Mk\in N(X)$ denotes the canonical class.
 
\item 
We call $F$ \MdefAttr{minimal}{family} if $F$ is a rational covering family and of minimal degree 
\Mwrt all rational covering families of $X$.

\item 
We call $F$ \MdefAttr{complete}{family} if there exists a curve $C\subset X$
\Mst the set $\Mset{C'\subset X}{ [C']=[C] }$ defines exactly the set of members of $F$.
In other words, $F$ forms a complete linear series.

\end{itemize}
We denote the classes of minimal families of $X\subset \MbbP^n$ as 
\[
S(X,\Mh):=\Mset{ [F]\in N(X) }{ \Mh\cdot[F]=\theta,~ F \text{ is a rational covering family of } X  }, 
\] 
where $\theta:=\min \Mset{ \Mh\cdot[F]}{ F \text{ is a rational covering family of } X }$.
Recall that $\Mh$ is the class of hyperplane sections.
Notice that we assume that $F$ is real unless explicitly stated otherwise and thus $\sigma_*[F]=[F]$.

\subsection{}
\label{sec:chain}

Let $Y$ be the smooth model of a birationally ruled surface $X\subset \MbbP^n$ 
and let $\Mh\in N(X)$ denote the class of hyperplane sections.
We call $(Y,\Mh)$ a \Mdef{ruled pair}, if $\Mh$ is nef and big and if there are no exceptional curves
that are orthogonal to $\Mh$.
The \Mdef{linear normalization} of $X$ 
is defined as $\varphi_{\Mh}(Y)\subset\MbbP^m$ for $m\geq n$. 
Here $\varphi_{\Mh}$ denotes the map associated to the class of hyperplane sections $\Mh$ 
and $X$ is a linear projection of $\varphi_{\Mh}(Y)$.
If $h^0(\Mh+\Mk)>1$, then an \Mdef{adjoint relation} is defined as
\[
\Mrow{\mu}{(Y,\Mh)}{(Y',\Mh'):=(\mu(Y),\mu_*(\Mh+\Mk))},
\]
where $\Mrow{\mu}{Y}{Y'}$ is a birational morphism that contracts all exceptional curves $E\subset Y$
\Mst $(\Mh+\Mk)\cdot [E]=0$. 
It follows from \citep[Proposition~1]{nls-f6} that $(Y',\Mh')$ is again a ruled pair.

An \Mdef{adjoint chain} is a chain of subsequent adjoint relations
\[
\Marrow{(Y_0,\Mh_0)}{\mu_0}{(Y_1,\Mh_1)}\Marrow{}{\mu_1}{}\ldots\Marrow{}{\mu_{\ell-1}}{(Y_\ell,\Mh_\ell)},
\]
\Mst $h^0(\Mh_\ell+\Mk_\ell)\leq 1$.
We call $(Y_\ell,\Mh_\ell)$ a \Mdef{minimal ruled pair}.

\begin{remark}
The current definition for adjoint relation differs from \citep[Section~3]{nls-f6},
where instead of $h^0(\Mh+\Mk)>1$ we require that 
the nef threshold of $\Mh$ is positive with $\Mh\neq -\Mk$.
A posteriori both definitions are equivalent by \citep[Proposition~1]{nls-f6},
but the current characterization has the advantage of being computable.
\Mend
\end{remark}

\begin{proposition}
\textbf{(adjoint chain)}
\label{prp:chain}
\begin{itemize}[topsep=0pt]
\Mmclaim{a}
If $\Mrow{\mu}{(Y,\Mh)}{(Y',\Mh')}$ is an adjoint relation \Mst $\Mh'^2>0$,
then it is unique up to biregular isomorphism.
Moreover, the contracted exceptional curves are disjoint.

\Mmclaim{b} 
If $Y_0$ is the smooth model of a birationally ruled surface $X\subset \MbbP^n$,
then its adjoint chain exists.

\Mmclaim{c}
If $(Y_\ell,\Mh_\ell)$ is a minimal ruled pair, then either
one of the following holds:
\begin{enumerate}[topsep=0pt]
\item 
The surface $Y_\ell$ is a weak del Pezzo surface \Mst $\Mh_\ell=-\alpha\Mk_\ell$ for 
$\alpha\in\{1,\frac{1}{3},\frac{2}{3},\frac{1}{2}\}$.

\item
The surface $Y_\ell$ is a $\MbbP^1$-bundle \Mst the class of the 
fiber is either $-\frac{2}{\Mh_\ell\cdot\Mk_\ell}\Mh_\ell$ or 
$-\frac{2}{(2\Mh_\ell+\Mk_\ell)\cdot\Mk_\ell}(2\Mh_\ell+\Mk_\ell)$.

\end{enumerate}
\end{itemize}
\end{proposition}

\begin{proof}
Assertion \Mrefmclaim{a} follows from \citep[Lemma~1]{nls-f6}.
It follows from \citep[Proposition~3]{nls-f6} that \Mrefmclaim{b} holds.
Assertion \Mrefmclaim{c} is a consequence of \citep[Proposition~2]{nls-f6}.  
\end{proof}

The following lemma is an adaption of \citep[Theorem~4.6]{sil1} which is 
attributed to Comessatti. See also \citep[Theorem~1.9]{kol2}. 
We generalize the result to also 
include singular surfaces, as many basic surfaces that occur in 
geometric modelling are singular. For example the ring torus
is a weak del Pezzo surface and its linear normalization in $\MbbP^4$
has four complex conjugate isolated singularities.

\begin{lemma}
\textbf{(classification of rational real surfaces)}
\label{lem:basis}
\\
If $Y_0$ is a $\MbbC$-rational real surface, then
there exists an adjoint chain
$\Marrow{(Y_0,\Mh_0)}{\mu_0}{(Y_1,\Mh_1)}\Marrow{}{\mu_1}{}\ldots\Marrow{}{\mu_{\ell-1}}{(Y_\ell,\Mh_\ell)}$,
\Mst one of the following holds:
\begin{enumerate}[topsep=0pt]
\item 
The surface $Y_\ell$ is a blowup of $\MbbP^2$.
One has that $N(Y_0)\cong \Mmod{\Me_0,\Me_1,\ldots,\Me_r}$ and $N(Y_\ell)\cong \Mmod{\Me_0,\Me_1,\ldots,\Me_s}$ for $0\leq s\leq r$ and $s\leq 8$
so that $\sigma_*(\Me_0)=\Me_0$, $\sigma_*(\{\Me_1,\ldots,\Me_r\})=\{\Me_1,\ldots,\Me_r\}$ and $\Mk_0=-3\Me_0+\Me_1+\ldots+\Me_r$.
The class $\Me_0$ is the pullback of the class of lines in $\MbbP^2$.
The classes $\Me_j$ for $s< j\leq r$ are pullbacks of
classes of exceptional curves that are contracted by some $\mu_i$ for $0\leq i\leq \ell$.
If $\mu_{i*}(\Me_j)=0$ and $\mu_{i'*}(\Me_{j'})=0$ \Mst $j>j'$, then $i\leq i'$.

\item 
The surface $Y_\ell$ is a blowup of $\MbbP^1\times\MbbP^1$.
One has
$N(Y_0)\cong \Mmod{\Ml_0,\Ml_1,\Mp_1,\ldots,\Mp_r}$ and $N(Y_\ell)\cong \Mmod{\Ml_0,\Ml_1,\Mp_1,\ldots,\Mp_s}$ for $0\leq s\leq r$ and $s\leq 7$
so that $\sigma_*(\{\Ml_0,\Ml_1\})=\{\Ml_0,\Ml_1\}$, $\sigma_*(\{\Mp_1,\ldots,\Mp_r\})=\{\Mp_1,\ldots,\Mp_r\}$ and 
$\Mk_0=-2(\Ml_0+\Ml_1)+\Mp_1+\ldots+\Mp_r$.
The classes $\Ml_0$ and $\Ml_1$ are pullbacks of the class of the fiber of first and second projections of $\MbbP^1\times\MbbP^1$
\Mresp.
The classes $\Mp_j$ for $s< j\leq r$ are pullbacks of
classes of exceptional curves that are contracted by some $\mu_i$ for $0\leq i\leq \ell$.
If $\mu_{i*}(\Mp_j)=0$ and $\mu_{i'*}(\Mp_{j'})=0$ \Mst $j>j'$, then $i\leq i'$.

\item 
There exists $f\in S(Y_\ell,\Mh_\ell)$ \Mst $\Mk_\ell\cdot f=-2$ and $f^2=0$.

\item 
The surface $Y_0$ is not $\MbbR$-rational and $1\leq \Mk_\ell^2\leq 2$. 

\end{enumerate}
Moreover, in cases (i) and (ii) there are natural inclusions $\Mhookrow{\iota}{N(Y_i)}{N(Y_0)}$ for all $0<i\leq \ell$ 
and these inclusions preserve the generators of the bases.
\end{lemma}

\begin{proof}
It follows from \PRP{chain}.\Mrefmclaim{c} that $Y_\ell$ is either a weak del Pezzo surface 
or a $\MbbP^1$-bundle. In the latter case, the class $f$ of the fiber is 
in $S(Y_\ell,\Mh_\ell)$ with $\Mk_\ell\cdot f=-2$ as asserted in case 3. 
So we may assume \Mwlog that $Y_\ell$ is a weak del Pezzo surface so that $-\Mk_\ell$ is nef and big.

We apply the real minimal model program (MMP) to $Y_\ell$ and contract at each step
a real exceptional curve or disjoint complex conjugate exceptional curves \citep[Theorem~1.8]{kol2}.
By \citep[Proposition~8.1.23]{dol1} the result after each step is again a weak del Pezzo surface.

Let us first assume that $Y_\ell$ is a smooth del Pezzo surface so that $-\Mk$ is ample.
It follows from \citep[Theorem~1.9]{kol2} and \citep[Theorem~4.6]{sil1}
that at the end of MMP, we obtain a surface $Z$ \Mst either 
$Z\cong\MbbP^2$, $Z\cong\MbbP^1\times\MbbP^1$, $Z$ is a conic bundle, 
or $Z$ is not $\MbbR$-rational with $1\leq\Mk_\ell^2\leq 2$ as stated in case 4. 
If $Z$ is a conic bundle and $f\in N(Z)$ is the class of the fiber, 
then $\Mk\cdot f=-2$ by the arithmetic genus formula. 
Either $Z$ is geometrically ruled and 
we are in case 1 or 2 of the lemma or 
$f\in S(Y_\ell,\Mh_\ell)$ as in case 3.

Now suppose that $Y_\ell$ is a weak del Pezzo surface.
Thus $Y_\ell$ is over $\MbbC$ the blowup of the plane in at most 8 points.
The configuration of the centers of blowup are 
determined by the effective (-2)-classes in $N(Y_\ell)$ \citep[Section~8.2.7]{dol1}.
If the centers of blowup are in general position, then there are no effective (-2)-classes.
Suppose by contradiction that after applying the MMP we end up with a weak del Pezzo surface $Z'$ 
that is not covered by case 1, 2, 3 or 4. There are 3 possibilities:
\begin{itemize}[topsep=0pt,itemsep=0pt]
\item[(i)] 
There exists classes $e,e'\in N(Z')$ \Mst $\Mk\cdot e=\Mk\cdot e'=e^2=e'^2=-1$, $e\cdot e'=0$ and $\sigma_*(e)=e'$.
In the smooth del Pezzo scenario $e$ and $e'$ 
would be classes of disjoint complex conjugate exceptional curves that
can be contracted.
By \citep[Lemma~8.2.22]{dol1}, $e=a+c$ where $a$ is the class of 
an exceptional curve, $c$ is the sum of effective (-2)-classes, $c^2=-2$ and $a\cdot c=0$. 
Thus $e=a+c$ and $e'=a'+c'$, where $\sigma_*(c)=c'$ and  $a$, $a'$ are the classes of exceptional curves \Mst $\sigma_*(a)=a'$.
By Hodge index theorem and $\Mk\cdot( c+c')=0$ we have $(c+c')^2<0$ so that either $0\leq c\cdot c'\leq 1$ or $c=c'$.
Since $a$ and $a'$ are classes of lines in the anticanonical embedding we must have $0\leq a\cdot a'\leq 1$.
Notice that $c\neq c'$, otherwise $e\cdot e'=a\cdot a'-2=0$, which is impossible.
Therefore $a\cdot a'=0$, since $e\cdot e'=a\cdot a'+ c\cdot c'+ \Delta=0$ where $\Delta:=a\cdot c'+a'\cdot c\geq 0$.
We arrived at a contradiction, since $Z'$ is not at the end of the MMP
as the disjoint complex conjugate exceptional curves can be contracted. 

\item[(ii)] 
There exists a class $e\in N(Z')$ \Mst $\Mk\cdot e=e^2=-1$ and $\sigma_*(e)=e$. 
In the smooth del Pezzo scenario, $e$ would be the class of a real exceptional curve that
can be contracted. 
By the same arguments as in case (i), we find that there exist a 
real exceptional curve that can be contracted and thus we arrived at 
at contradiction as $Z'$ is not the end result of MMP.

\item[(iii)] 
There exists a class $u\in N(Z')$ \Mst $\Mk\cdot u=-2$ and $u^2=0$. 
In the smooth del Pezzo scenario $u$ would be the class of a conic in a bundle.
Using the same method as in \citep[Lemma~8.2.22]{dol1}, we can show that 
$u=b+c$ where $b$ is class of conic \Mst $\Mk\cdot b +2=b^2=0$ and $c$ is the sum of effective (-2)-classes.
By Riemann-Roch theorem and Kawamata-Viehweg vanishing theorem we find that 
$h^0(b)=2$. Thus we arrived at a contradiction, since $Z'$ is a conic bundle
as covered in case 3.
\end{itemize}
For the remaining assertions for cases 1 and 2
we recall that the contracted exceptional 
curves with class $\Me_i$ for some $i$ are orthogonal by \PRP{chain}.\Mrefmclaim{a}.
We choose an indexing of the exceptional classes so that classes with 
a lower index are contracted later in the chain.
The specification of the canonical class $\Mk_0$ follows from 
\citep[(1.41)]{deb1}.
The remaining details for these cases are now 
straightforward and left to the reader.
\end{proof}

\newpage
\begin{definition}
\textbf{\textit{(adjoint chains of type 1 and type 2)}}
\label{def:chain}
\\
An adjoint chain
is of 
\MdefAttr{type 1}{adjoint chain} or
\MdefAttr{type 2}{adjoint chain},
if 
\LEM{basis}.1 
or \LEM{basis}.2 holds \Mresp.
Similarly, we say that an adjoint relation $\Mrow{\mu_0}{(Y_0,\Mh_0)}{(Y_1,\Mh_1)}$
or surface $Y_0$ is of 
\MdefAttr{type 1}{adjoint relation} or
\MdefAttr{type 2}{adjoint relation}, if it is part of 
an adjoint chain of type 1 or 2 \Mresp.
\Mend
\end{definition}

\begin{lemma}
\textbf{(coordinates of classes)}
\label{lem:coord}
\\
We consider adjoint chain
$\Marrow{(Y_0,\Mh_0)}{\mu_0}{(Y_1,\Mh_1)}\Marrow{}{\mu_1}{}\ldots\Marrow{}{\mu_{\ell-1}}{(Y_\ell,\Mh_\ell)}$.
Let $f_0\in N(Y_0)$ be the class of a curve and define 
$f_i:=(\mu_{i-1}\circ\ldots\circ\mu_0)_*(f_0)$
for $0< i\leq \ell$.
Let $[t]:=1$ if $t\geq 0$ and $[t]:=0$ if $t<0$ for all $t\in\MbbZ$.
\begin{itemize}[topsep=0pt,itemsep=0pt]
\Mmclaim{a}
If the adjoint chain is of type 1, then
\begin{align*}
\Mh_i &= (\alpha_0-3i)\Me_0-[\alpha_1-i](\alpha_1-i)\Me_1-\ldots-[\alpha_r-i](\alpha_r-i)\Me_r,
\\
\Mk_i &= -3\Me_0+[\alpha_1-i]\Me_1+\ldots+[\alpha_r-i]\Me_r, 
\\
f_i   &=
\beta_0\Me_0-[\alpha_1-i]\beta_1\Me_1-\ldots-[\alpha_r-i]\beta_r\Me_r,
\end{align*}
where
$\alpha_0-3i>0$ for $0\leq i\leq \ell$, 
$\beta_j\geq 0$ for $0\leq j\leq r$
and 
$\alpha_t\geq\alpha_{t+1} >0$ for $0< t< r$.

\Mmclaim{b}
If the adjoint chain is of type 2, then
\begin{align*}
\Mh_i &= (\alpha_0-3i)(\Ml_0+\Ml_1)-[\alpha_1-i](\alpha_1-i)\Mp_1-\ldots-[\alpha_r-i](\alpha_r-i)\Mp_r,
\\
\Mk_i &= -2(\Ml_0+\Ml_1)+[\alpha_1-i]\Mp_1+\ldots+[\alpha_r-i]\Mp_r, 
\\ 
f_i   &= \gamma_0\Ml_0+\gamma_1\Ml_1-[\alpha_1-i]\beta_1\Mp_1-\ldots-[\alpha_r-i]\beta_r\Mp_r,
\end{align*}
where
$\alpha_0-3i>0$ for $0\leq i\leq \ell$, 
$\gamma_0,\gamma_1 \geq 0$,
$\beta_j>0$ for $1\leq j\leq r$
and 
$\alpha_t\geq\alpha_{t+1} >0$ for $0< t< r$.
\end{itemize}
\end{lemma}

\begin{proof}
The proof is straightforward and left to the reader.
\end{proof}

\begin{example}
\textbf{\textit{(adjoint chain)}}
\label{exm:chain}
\\
We consider the following adjoint chain
\[
\Marrow{(Y_0,\Mh_0)}{\mu_0}{(Y_1,\Mh_1)}
\Marrow{}{\mu_1}{}  
\ldots
\Marrow{}{\mu_5}{(Y_6,\Mh_6)},
\]
where 
\[
\begin{array}{r@{\,=\,}r@{\,}r@{\,}r@{\,}r@{\,}r@{\,}r@{\,}r@{\,}r@{\,}r@{\,}r@{\,}r@{\,}r@{\,}r@{\,}r@{\,}r@{\,}r@{\,}r}
\Mh_0 & 19\Me_0 &-& 6\Me_1 &-& 6\Me_2 &-& 4\Me_3 &-& 4\Me_4 &-& 3\Me_5 &-& 3\Me_6 &-& 2\Me_7 &-& 2\Me_8   \\
\Mh_1 & 16\Me_0 &-& 5\Me_1 &-& 5\Me_2 &-& 3\Me_3 &-& 3\Me_4 &-& 2\Me_5 &-& 2\Me_6 &-&  \Me_7 &-&  \Me_8   \\
\Mh_2 & 13\Me_0 &-& 4\Me_1 &-& 4\Me_2 &-& 2\Me_3 &-& 2\Me_4 &-&  \Me_5 &-&  \Me_6 & &        & &          \\
\Mh_3 & 10\Me_0 &-& 3\Me_1 &-& 3\Me_2 &-&  \Me_3 &-&  \Me_4 & &        & &        & &        & &          \\
\Mh_4 &  7\Me_0 &-& 2\Me_1 &-& 2\Me_2 & &        & &        & &        & &        & &        & &          \\
\Mh_5 &  4\Me_0 &-&  \Me_1 &-&  \Me_2 & &        & &        & &        & &        & &        & &          \\
\Mh_6 &  3\Me_0 & &        & &        & &        & &        & &        & &        & &        & &          \\
\end{array}
\]
The action of the real structure on $N(Y_0)$ is defined by $\sigma_*(\Me_0)=\Me_0$ and $\sigma_*(\Me_i)=\Me_{i+1}$ for $i\in\{1,3,5,7\}$.
Notice that $Y_6\cong\MbbP^2$ and $\Mh_6=-\Mk_6$.
Let $f_0:=2\Me_0-\Me_1-\Me_2-\Me_3-\Me_4$ be the class of the pullback of conics in the plane, through 4 points.
By \LEM{coord} we have $f_0=f_1=f_2=f_3$, $f_4=f_5=2\Me_0-\Me_1-\Me_2$ and $f_6=2\Me_0$.

We will prove in \SEC{min} that $S(Y_0,\Mh_0)=\{~f_0~\}$, $S(Y_1,\Mh_1)=\{~\Me_0,~f_1~\}$ 
and $S(Y_i,\Mh_i)=\{~\Me_0~\}$ for $6\leq i\leq 2$, where $\Me_0$ and $f_1$ are 
classes of one- and two-dimensional minimal families \Mresp.

Over the complex numbers or if the real structure would act as the identity, 
$\Me_0-\Me_1$ and $\Me_0-\Me_2$ 
are classes of minimal families of $Y_i$ for $0\leq i<6$ (see \RMK{complex}).
\Mend
\end{example}

\section{Classification of minimal families}
\label{sec:min}

\subsection{}

In the following lemma we recall the proof of \citep[Proposition ~20]{nls1}
for the convenience of the reader.

\begin{lemma}
\textbf{(canonical degree of families)}
\label{lem:kf}
\\
If $f\in N(Y)$ is the class of a rational family, then $\Mk \cdot f\leq -2$.
\end{lemma}

\begin{proof}
Let $f\in N(Y)$ be the class of the rational family defined by the divisor $F\subset Y\times B$.
Notice that if $F$ is complete, then $p_a(f)=0$ so that $\Mk\cdot f=-2-f^2\leq -2$.
If $F$ is not complete, then its members are necessarily singular.
We define $\Mrow{\alpha}{Z}{F}$ to be a birational morphism such that $Z$ is a 
smooth model of $F$.
Suppose that $F_b\subset Y$ is a member of $F$, for some $b\in B$.
Let $G_b\subset Z$ be the pullback of $F_b$ along the morphism $\alpha$ composed with
the first projection $\Mrow{\pi_1}{F}{Y}$.
Thus $G\subset Z\times B$ defines a family of $Z$, whose members are~$G_b$ for $b\in B$.
By the pullback formula for divisors we have $\Mk_Z=(\pi_1\circ\alpha)^*\Mk+r$,
where $r\in N(Z)$ is the ramification divisor and $\Mk$ is the canonical class on $Y$.
The following equality relates the canonical degrees of $F$ and $G$: 
\[
\Mk \cdot f=\Mk \cdot (\pi_1\circ\alpha)_* [G]=(\pi_1\circ\alpha)^*\Mk\cdot [G]=\Mk_Z\cdot [G]-r\cdot [G]. 
\]
If $a,b\in B$ \Mst $a\neq b$, then $F_a\neq F_b$ and thus $G_a\cap G_b=\emptyset$,
since $\pi_1\circ\alpha$ is a birational morphism. It follows that $[G]^2=0$. 
By definition, $G_b$ is via
$\pi_1\circ\alpha$ birationally equivalent to $F_b$ and thus $G_b$ is rational, for all~$b\in B$.
We observe that $G_b$ is a fiber of $\Mrow{\pi_2\circ\alpha}{Z}{B}$, for all~$b\in B$.
As a consequence of Sard's theorem, the general fiber is smooth and thus $p_a([G])=0$.
We apply the arithmetic genus formula and we find that $\Mk_Z\cdot [G]=-2-[G]^2=-2$.
By \citep[(1.41)]{deb1} we know that $h^0(r)>0$ and since 
$G\subset Z\times B$ defines a family whose members are irreducible and 
movable, it follows that $G$ is nef so that $[G]\cdot r\geq 0$.
This concludes the proof since $\Mk \cdot f=\Mk_Z\cdot [G]-r\cdot [G]\leq -2$.
\end{proof}

\begin{lemma}
\textbf{(self-intersection of classes of rational curves)} 
\label{lem:self}
\\
If $Y$ is the smooth model of a weak del Pezzo surface 
and $f\in S(Y,-\Mk)$ \Mst $\Mk\cdot f=-2$,
then $f^2\in\{0,2,4\}$.
\end{lemma}

\begin{proof}
By the adjunction formula one has $p_a(f)=\frac{1}{2}f^2$ and thus $f^2$ is even and nonnegative.
If $f^2\geq 2$, then 
by Hodge index theorem and $f\cdot(f+\alpha\Mk)=0$ for some $\alpha>0$,
it follows that either $f=-\alpha\Mk$ or $(f+\alpha\Mk)^2<0$.
If $(f+\alpha\Mk)^2<0$, then $f^2<\alpha(4-\alpha\Mk^2)$ so that $f^2<4$.
If $f=-\alpha\Mk$ with $\Mk\cdot f=-2$, then either $f=-\Mk$ 
with $\Mk^2=2$ --- or --- $f=-2\Mk$ with $\Mk^2=1$. 
We conclude that $f^2\in\{0,2,4\}$ as asserted.
\end{proof}

\begin{definition}
\textbf{($\Psi_0$, $\Psi_1$, $\Psi_2$ and $\Psi_4$)}
\label{def:psi}
\\
We define $\Psi_0$ to be the maximal set of classes in a NS-lattice of infinite rank \Mst each 
class in $\Psi_0$ is \Mwrt a type 1 basis --- up to permutation of the $(\Me_i)_{i>0}$ --- in the following list:
\begin{center}
{\tiny
\begin{tabular}{r@{ }c@{ }r@{ }c@{ }r@{ }c@{ }r@{ }c@{ }r@{ }c@{ }r@{ }c@{ }r@{ }c@{ }r@{ }c@{ }r}
$  \Me_0$ & $-$ & $ \Me_1$,&     &        &     &        &     &        &     &        &     &        &     &        &     &        \\
$ 2\Me_0$ & $-$ & $ \Me_1$ & $-$ & $ \Me_2$ & $-$ & $ \Me_3$ & $-$ & $ \Me_4$,&     &        &     &        &     &        &     &        \\
$ 3\Me_0$ & $-$ & $2\Me_1$ & $-$ & $ \Me_2$ & $-$ & $ \Me_3$ & $-$ & $ \Me_4$ & $-$ & $ \Me_5$ & $-$ & $ \Me_6$,&     &        &     &        \\
$ 4\Me_0$ & $-$ & $2\Me_1$ & $-$ & $2\Me_2$ & $-$ & $2\Me_3$ & $-$ & $ \Me_4$ & $-$ & $ \Me_5$ & $-$ & $ \Me_6$ & $-$ & $ \Me_7$,&     &        \\
$ 5\Me_0$ & $-$ & $2\Me_1$ & $-$ & $2\Me_2$ & $-$ & $2\Me_3$ & $-$ & $2\Me_4$ & $-$ & $2\Me_5$ & $-$ & $2\Me_6$ & $-$ & $ \Me_7$,&     &        \\
$ 4\Me_0$ & $-$ & $3\Me_1$ & $-$ & $ \Me_2$ & $-$ & $ \Me_3$ & $-$ & $ \Me_4$ & $-$ & $ \Me_5$ & $-$ & $ \Me_6$ & $-$ & $ \Me_7$ & $-$ & $ \Me_8$,\\
$ 5\Me_0$ & $-$ & $3\Me_1$ & $-$ & $2\Me_2$ & $-$ & $2\Me_3$ & $-$ & $2\Me_4$ & $-$ & $ \Me_5$ & $-$ & $ \Me_6$ & $-$ & $ \Me_7$ & $-$ & $ \Me_8$,\\
$ 6\Me_0$ & $-$ & $3\Me_1$ & $-$ & $3\Me_2$ & $-$ & $2\Me_3$ & $-$ & $2\Me_4$ & $-$ & $2\Me_5$ & $-$ & $2\Me_6$ & $-$ & $ \Me_7$ & $-$ & $ \Me_8$,\\
$ 7\Me_0$ & $-$ & $3\Me_1$ & $-$ & $3\Me_2$ & $-$ & $3\Me_3$ & $-$ & $3\Me_4$ & $-$ & $2\Me_5$ & $-$ & $2\Me_6$ & $-$ & $2\Me_7$ & $-$ & $ \Me_8$,\\
$ 7\Me_0$ & $-$ & $4\Me_1$ & $-$ & $3\Me_2$ & $-$ & $2\Me_3$ & $-$ & $2\Me_4$ & $-$ & $2\Me_5$ & $-$ & $2\Me_6$ & $-$ & $2\Me_7$ & $-$ & $2\Me_8$,\\
$ 8\Me_0$ & $-$ & $3\Me_1$ & $-$ & $3\Me_2$ & $-$ & $3\Me_3$ & $-$ & $3\Me_4$ & $-$ & $3\Me_5$ & $-$ & $3\Me_6$ & $-$ & $3\Me_7$ & $-$ & $ \Me_8$,\\
$ 8\Me_0$ & $-$ & $4\Me_1$ & $-$ & $3\Me_2$ & $-$ & $3\Me_3$ & $-$ & $3\Me_4$ & $-$ & $3\Me_5$ & $-$ & $2\Me_6$ & $-$ & $2\Me_7$ & $-$ & $2\Me_8$,\\
$ 9\Me_0$ & $-$ & $4\Me_1$ & $-$ & $4\Me_2$ & $-$ & $3\Me_3$ & $-$ & $3\Me_4$ & $-$ & $3\Me_5$ & $-$ & $3\Me_6$ & $-$ & $3\Me_7$ & $-$ & $2\Me_8$,\\
$10\Me_0$ & $-$ & $4\Me_1$ & $-$ & $4\Me_2$ & $-$ & $4\Me_3$ & $-$ & $4\Me_4$ & $-$ & $3\Me_5$ & $-$ & $3\Me_6$ & $-$ & $3\Me_7$ & $-$ & $3\Me_8$,\\
$11\Me_0$ & $-$ & $4\Me_1$ & $-$ & $4\Me_2$ & $-$ & $4\Me_3$ & $-$ & $4\Me_4$ & $-$ & $4\Me_5$ & $-$ & $4\Me_6$ & $-$ & $4\Me_7$ & $-$ & $3\Me_8$.\\
\end{tabular}
}
\end{center}
We define $\Psi_2$ to be the maximal set of classes in a NS-lattice of infinite rank \Mst each 
class in $\Psi_2$ is \Mwrt a type 1 basis --- up to permutation of the $(\Me_i)_{i>0}$ --- in the following list:
\begin{center}
{\tiny
\begin{tabular}{r@{ }c@{ }r@{ }c@{ }r@{ }c@{ }r@{ }c@{ }r@{ }c@{ }r@{ }c@{ }r@{ }c@{ }r@{ }c@{ }r}
$ 3\Me_0$ & $-$ & $ \Me_1$ & $-$ & $ \Me_2$ & $-$ & $ \Me_3$ & $-$ & $ \Me_4$ & $-$ & $ \Me_5$ & $-$ & $ \Me_6$ & $-$ & $ \Me_7$,&     &          \\
$ 4\Me_0$ & $-$ & $2\Me_1$ & $-$ & $2\Me_2$ & $-$ & $ \Me_3$ & $-$ & $ \Me_4$ & $-$ & $ \Me_5$ & $-$ & $ \Me_6$ & $-$ & $ \Me_7$ & $-$ & $ \Me_8$,\\
$ 5\Me_0$ & $-$ & $2\Me_1$ & $-$ & $2\Me_2$ & $-$ & $2\Me_3$ & $-$ & $2\Me_4$ & $-$ & $ \Me_5$ & $-$ & $ \Me_6$ & $-$ & $ \Me_7$ & $-$ & $ \Me_8$,\\
$ 6\Me_0$ & $-$ & $3\Me_1$ & $-$ & $2\Me_2$ & $-$ & $2\Me_3$ & $-$ & $2\Me_4$ & $-$ & $2\Me_5$ & $-$ & $2\Me_6$ & $-$ & $2\Me_7$ & $-$ & $ \Me_8$,\\
$ 7\Me_0$ & $-$ & $3\Me_1$ & $-$ & $3\Me_2$ & $-$ & $3\Me_3$ & $-$ & $2\Me_4$ & $-$ & $2\Me_5$ & $-$ & $2\Me_6$ & $-$ & $2\Me_7$ & $-$ & $2\Me_8$,\\
$ 8\Me_0$ & $-$ & $3\Me_1$ & $-$ & $3\Me_2$ & $-$ & $3\Me_3$ & $-$ & $3\Me_4$ & $-$ & $3\Me_5$ & $-$ & $3\Me_6$ & $-$ & $2\Me_7$ & $-$ & $2\Me_8$,\\
$ 9\Me_0$ & $-$ & $4\Me_1$ & $-$ & $3\Me_2$ & $-$ & $3\Me_3$ & $-$ & $3\Me_4$ & $-$ & $3\Me_5$ & $-$ & $3\Me_6$ & $-$ & $3\Me_7$ & $-$ & $3\Me_8$,\\
\end{tabular}
}
\end{center}
We define 
$\Psi_1:=\{~ \Me_0 ~\}$ and $\Psi_4:=\{~ 6\Me_0-2\Me_1-\ldots-2\Me_8 ~\}$. 

Notice that $f^2=\alpha$ for all $f\in \Psi_\alpha$.

If an NS-lattice is of rank at least 3,
then we consider the following basis changes between bases of type 1 and type 2:
\begin{align}
\label{eqn:basis}
(\Ml_0,\Ml_1,\Mp_1,\ldots,\Mp_r)\mapsto&(\Ml_0+\Ml_1-\Mp_1,\Ml_0-\Mp_1,\Ml_1-\Mp_1,\Mp_2,\ldots,\Mp_r),
\\
(\Me_0,\Me_1,\ldots,\Me_r)\mapsto&(\Me_0-\Me_2,\Me_0-\Me_1, \Me_0-\Me_1-\Me_2,\Me_3,\ldots,\Me_r).\nonumber
\end{align}
If, for example, $N(Y)$ is a lattice with basis of type 2 and $f\in N(Y)$ \Mst $f=\Ml_0+\Ml_1-\Mp_1-\Mp_2$, then
$f \in \Psi_0$, since we will implicitly use \EQN{basis} after which $f$ is equal to $\Me_0-\Me_3$.
\Mend
\end{definition}

\begin{lemma}
\textbf{(classes of minimal families)} 
\label{lem:psi}
\\
If $Y$ is the smooth model of a weak del Pezzo surface 
and $f\in S(Y,-\Mk)$ \Mst $\Mk\cdot f=-2$,
then $f\in \Psi_0\cup\Psi_2\cup\Psi_4$.
\end{lemma}

\begin{proof}
The set $\Psi_\alpha$ is the output of 
\citep[Algorithm~1]{nls-algo-fam}
with input $-\Mk\cdot f=2$ and $f^2=\alpha$.
By \LEM{self} we have $\alpha\in\{0,2,4\}$
which concludes this proof.
\end{proof}

\begin{lemma}
\textbf{(rational curves with class in $\Psi_0$)}
\label{lem:psi0}
\\
Suppose that $Y$ is the smooth model of a weak del Pezzo surface \Mst $1\leq \Mk^2 \leq 6$.
We have that $\Mk\cdot f=-2$, $f^2=0$ and $h^0(f)=2$ for all $f\in \Psi_0$.
Moreover, if $Y$ is $\MbbR$-rational, then $S(Y,-\Mk)\cap \Psi_0\neq \emptyset$.
\end{lemma}

\begin{proof}
We call \citep[Algorithm~1]{nls-algo-fam}
with input $-\Mk\cdot f=2$ and $f^2=0$ and obtain $\Psi_0$ as output. 
It follows from  
Riemman-Roch theorem and Serre duality that $h^0(f)=2$ for all $f\in \Psi_0$.
By the arithmetic genus formula one has $p_a(f)=0$.

Now suppose that $Y$ is $\MbbR$-rational
so that either case (1), (2) or (3) of \LEM{basis} holds.
If case (3) holds, then $S(Y,-\Mk)\cap \Psi_0\neq \emptyset$ as asserted.
If case (1) or (2) holds, then $Y$ is the real blowup of 
$\MbbP^2$ or $\MbbP^1\times\MbbP^1$ in at least 3 and 2 points \Mresp. 
Thus $4\leq \Mrnk(N(Y))\leq 9$ and 
there exists $f\in \Psi_0$ such that, up to permutation,
$f\in\{~2\Me_0-\Me_1-\Me_2-\Me_3-\Me_4,~\Ml_0+\Ml_1-\Mp_1-\Mp_2\}$
and $\sigma_*(f)=f$.

If $f\in\Psi_0$ is the class of an irreducible curve \Mst $\sigma_*(f)=f$, 
then the lemma holds, since $-\Mk\cdot f\geq 2$ by \LEM{kf}.

Suppose that $f\in\{~2\Me_0-\Me_1-\Me_2-\Me_3-\Me_4,~\Ml_0+\Ml_1-\Mp_1-\Mp_2\}$
\Mst $\sigma_*(f)=f$, and $f$ decomposes into a moving and a fixed component. 
We consider up to permutation, all possible decompositions of $f$.
If $f=2\Me_0-\Me_1-\Me_2-\Me_3-\Me_4$ and $h^0(\Me_0-\Me_2-\Me_3-\Me_4)=1$, 
then the moving component $\Me_0-\Me_1\in \Psi_0$ is the class of a rational family.
Since $-\Mk$ is nef one has that $h^0(\Me_0-\Me_1-\Me_2-\Me_3-\Me_4)=0$
and thus $\Me_0$ is not the moving component.
If $f=\Ml_0+\Ml_1-\Mp_1-\Mp_2$ and $h^0(\Ml_0-\Mp_1-\Mp_2)=1$, 
then $\Ml_1\in \Psi_0$ is the class of a rational family.
This concludes the proof of this lemma, as we considered all 
possible moving components.
\end{proof}

\begin{lemma}
\textbf{(rational curves with class $-\Mk\in\Psi_2$ where $\Mk^2=2$)}
\label{lem:psi2a}
\\
Suppose that $Y$ is the smooth model of a weak del Pezzo surface \Mst $\Mk^2=2$. 
If $C\subset Y$ is a rational curve with class $-\Mk$,
then 
$C$ is singular and the pullback of a tangent line 
of the quartic branching curve $B\subset \MbbP^2$
along the 2:1 morphism $\Mrow{\varphi_{-\Mk}}{Y}{\MbbP^2}$.
\end{lemma}

\begin{proof}
It follows from \citep[Theorem~8.3.2]{dol1} that $\varphi_{-\Mk}$ is a 2:1 covering.
Curves in the linear series $|-\Mk|$ are the pullback of lines in $\MbbP^2$.
Since $p_a(-\Mk)=1$, rational curves in $|-\Mk|$ have a singularity at the 
ramification locus, as discussed in \citep[Section~6.3.3]{dol1}.
It follows that rational curves are the pullback of tangent lines of 
the branching curve $B$.
\end{proof}

\begin{lemma}
\textbf{(rational curves with class $-2\Mk\in\Psi_4$ where $\Mk^2=1$)} 
\label{lem:psi4}
\\
Suppose that $Y$ is the smooth model of a weak del Pezzo surface \Mst $\Mk^2=1$. 
If $C\subset Y$ is a rational curve with class $-2\Mk$,
then 
$C$ is singular and the pullback of a bitangent plane of 
the sextic branching curve $B\subset Q$ --- intersected 
with the quadric cone $Q\subset\MbbP^3$ --- along the 2:1 morphism $\Mrow{\varphi_{-2\Mk}}{Y}{Q}$.
\end{lemma}

\begin{proof}
We know from \citep[Theorem~8.3.2]{dol1} that $\varphi_{-2\Mk}$ is a 2:1 covering.
Curves in the linear series $|-2\Mk|$ are the pullback of plane sections of $Q$.
Since $p_a(-\Mk)=2$, rational curves in $|-2\Mk|$ are singular at the ramification locus \Mst
that the delta invariants of the singularities adds up to two. 
It follows that $C$ is the pullback of a plane section of $Q$ that is bitangent to 
the branching curve $B$, as discussed in \citep[Section~8.8.1]{dol1}.
\end{proof}

\begin{lemma}
\textbf{(rational curves with class $-\Mk+e\in\Psi_2$ where $\Mk^2=1$)} 
\label{lem:psi2b}
\\
Suppose that $Y$ is the smooth model of a weak del Pezzo surface \Mst $\Mk^2=1$. 
If $C\subset Y$ is a rational curve with class $-\Mk+e$ where $\Mk\cdot e=e^2=-1$,
then $C$ is singular and the pullback of a tangent line 
of the quartic branching curve $B\subset \MbbP^2$
along the 2:1 morphism $\Mrow{\varphi_{-\Mk+e}}{Y}{\MbbP^2}$.
\end{lemma}

\begin{proof}
Since $C$ is a rational curve, its class $-\Mk+e$ has no fixed components. 
Therefore there exists no $c\in N(Y)$ \Mst $h^0(c)>0$, $c^2=-2$, $\Mk\cdot c=0$  
and $(-\Mk+e)\cdot c=e\cdot c<0$.
It now follows from \citep[Lemma~8.2.22]{dol1} that $e$ is the class of an exceptional curve $E$. 
Let $\Mrow{\pi}{Y}{Y'}$ be the contraction of $E$ to a smooth point in $Y'$.
It follows from \citep[Proposition~8.1.23]{dol1} that $Y'$ is a 
weak del Pezzo surface \Mst $\Mk'^2=2$ and $\Mk=\pi^*\Mk'+e$.
Thus $\pi(C)$ is a rational curve with class $-\Mk'$
and the pullback of a
hyperplane section along $\varphi_{-\Mk'}\circ \pi$ has class $-\Mk+e$.
This lemma is now a consequence of \LEM{psi2a}.
\end{proof}

\begin{proposition}
\textbf{(minimal family classes for minimal ruled pairs)}
\label{prp:mp} 
\\
Let $\Psi_0$, $\Psi_1$, $\Psi_2$ and $\Psi_4$ be as in \DEF{psi}.
If $(Y,\Mh)$ is a minimal ruled pair 
\Mst either $Y$ is $\MbbR$-rational or there exists $f\in S(Y,\Mh)$ \Mst $\Mk\cdot f=-2$,
then one of the following holds:
\begin{enumerate}[topsep=0pt]

\item 
$\Mh=-\Mk$, $1\leq \Mk^2\leq 2$, $S(Y,\Mh)\subset \Psi_0 \cup\Psi_2\cup \Psi_4$.

\item 
$\Mh=-\Mk$, $3\leq \Mk^2\leq 6$, $S(Y,\Mh)\subset \Psi_0$.

\item 
$\Mh=-\Mk$, $\Mk^2=7$, $S(Y,\Mh)\subset\Psi_0\cup\Psi_1$.

\item 
$\Mh=-\frac{\alpha}{2}\Mk$ for $1\leq \alpha\leq  2$, 
$\Mk^2=8$, 
$S(Y,\Mh)\subset\Psi_0\cup\{~\Ml_0,~\Ml_1,~ \Ml_0+\Ml_1 ~\}$.

\item 
$\Mh=-\frac{\alpha}{3}\Mk$ for $1\leq \alpha\leq  3$, 
$\Mk^2=9$, 
$S(Y,\Mh)=\Psi_1$.

\item 
$\Mh^2=0$, $\Mh\neq 0$,
$S(Y,\Mh)=\left\{~ -\frac{2}{\Mh\cdot\Mk}\Mh ~\right\}$.

\item 
$(2\Mh+\Mk)^2=0$, $2\Mh+\Mk\neq 0$,
$S(Y,\Mh)=\left\{~ -\frac{2}{(2\Mh+\Mk)\cdot\Mk}(2\Mh+\Mk) ~\right\}$.

\end{enumerate}
The following table summarizes attributes of a minimal families:
\begin{center}
\begin{tabular}{|c||c|c|c|c|c|}
\hline
class is in $\rightarrow$  & $\Psi_0$ & $\Psi_1$ & $\Psi_2$ & $\Psi_4$ & $\{\Ml_0+\Ml_1\}$ \\
\hline             
complete     & yes      & yes      & no       & no       & yes               \\
canonical degree & $-2$     & $-3$     & $-2$     & $-2$     & $-4$              \\
dimension    & $1$      & $2$      & $1$      & $1$      & $3$               \\
\hline
\end{tabular}
\end{center}
If $f$ is the class of an incomplete minimal family, then the curves in
the family are singular and the pullback 
of hyperplane sections, along the associated map $\varphi_f$,  
that are either tangent or bitangent to the branching locus.
\end{proposition}

\begin{proof}
It follows from \PRP{chain}.\Mrefmclaim{c} that (1-7) treat all possible cases.
The characterization of minimal families in the table and the assertion at the end,
follow from \LEM{psi}, \LEM{psi0}, \LEM{psi2a}, \LEM{psi4} and \LEM{psi2b}.

If there exists $f\in S(Y,\Mh)$ \Mst $\Mk\cdot f=-2$, then 
cases (1), (2), (3) and (4) are a direct consequence of \LEM{psi}.
In cases (6) and (7) the surface $Y$ is a projective line bundle (possibly $\MbbC$-irrational)
such that the fibers define the unique minimal family with canonical degree $-2$.

In the remainder of the proof we assume that $Y$ is $\MbbR$-rational.
Cases (1) and (2) follow from \LEM{psi0}.
Cases (6) and (7) are straightforward.
In cases (3), (4) and (5) the surface $Y$ is by \LEM{basis}
either $\MbbP^2$, the blowup of $\MbbP^2$ in a real point, 
the blowup of $\MbbP^2$ in two real points,
the blowup of $\MbbP^2$ in two complex conjugate points,
$\MbbP^1\times\MbbP^1$ or 
the blowup of $\MbbP^1\times\MbbP^1$ in a real point. 
Notice that the real structure $\sigma$ may flip the components of 
the fiber product $\MbbP^1\times\MbbP^1$.
The details for these cases are straightforward and left to the reader.
\end{proof}


\subsection{}

\begin{lemma}
\textbf{(pullback of classes of families)} 
\label{lem:pull}
\\
Let $\Mrow{\mu}{(Y,\Mh)}{(Y',\Mh')}$ be an adjoint relation.
If $f'\in N(Y')$ and $f:=\mu^*f'$,
then $\Mh\cdot f=\Mh'\cdot f' - \Mk'\cdot f'$.
Moreover, if $f'\in N(Y')$ is the class of a rational family \Mst $\Mk'\cdot f'=-2$, 
then $f$ is the class of a rational family as well.
\end{lemma}

\begin{proof}
We have $\Mk\cdot f=\Mk'\cdot f'$, since $\Mk\cdot f = \Mk\cdot \mu^*f'=\mu_*\Mk\cdot f'=\Mk'\cdot f'$
by the projection formula for divisor classes. 
Suppose that $\Mrow{\mu}{Y}{Y'}$ contracts exceptional curves with classes $(\Me_i)_{i\in I}$.
The contracted exceptional curves
are orthogonal to $\Mh+\Mk$ and thus $\mu^*\Mh'=\Mh+\Mk$ so that
\begin{equation*}
\Mh\cdot f = (\Mh+\Mk)\cdot f - \Mk\cdot f = \mu^*\Mh'\cdot f - \Mk\cdot f 
=  \Mh'\cdot f' - \Mk'\cdot f'.
\end{equation*} 
If $\Mtlf$ is the strict transform of $f'$ along $\mu$ and $\Mk'\cdot f'=-2$,
then $f=\Mtlf$, because $\Mk\cdot \Mtlf\leq -2$ by \LEM{kf}, $\Mk\cdot \Me_i=-1$ for $i\in I$ and
\[
  \Mk'\cdot f'
= \mu_*\Mk\cdot f'
= \Mk\cdot \mu^*f'
= \Mk\cdot(\Mtlf+\Msum{i\in I}{}m_i\Me_i)
= \Mk\cdot\Mtlf-\Msum{i\in I}{}m_i=-2.
\]
Therefore $m_i=0$ for all $i\in I$ and thus if $f'$ is a rational family, then $f$ is as well. 
\end{proof}

\begin{lemma}
\textbf{(pushforward of classes of families)}
\label{lem:push}
\\
Let $\Mrow{\mu}{(Y,\Mh)}{(Y',\Mh')}$ be an adjoint relation.
If $f\in N(Y)$ and $f':=\mu_*f$,
then $\Mh'\cdot f'=\Mh\cdot f + \Mk\cdot f$.
Moreover, if $f$ the class of a rational family, then $f'$
is the class of a rational family.
\end{lemma}

\begin{proof}
Since the classes of the exceptional curves that are contracted by 
$\mu$ are orthogonal to $\Mh+\Mk$, we have that $\mu^*\Mh'=\Mh+\Mk$.
It follows from the projection formula for classes that 
$
\Mh'\cdot f'
=
\mu^*\Mh'\cdot f
=
\Mh\cdot f + \Mk\cdot f.
$
This lemma is now a direct consequence of
the birational invariance of the geometric genus.
\end{proof}

\begin{proposition}
\textbf{(minimal families along adjoint relation)}
\label{prp:pp} 
\\
Let $\Mrow{\mu}{(Y,\Mh)}{(Y',\Mh')}$ be an adjoint relation.
\begin{itemize}[topsep=0pt]

\Mmclaim{a}
If $g'\in N(Y')$ is the class of a minimal family \Mst $\Mk'\cdot g'=-2$, then 
its pullback $g:=\mu^*g'$ is the class of a minimal family so that
$\Mk\cdot g=-2$ and $\Mh\cdot g=\Mh'\cdot g' +2$. 

\Mmclaim{b}
If $f\in N(Y)$ and $g'\in N(Y')$ are classes of minimal families \Mst $\Mk'\cdot g'=-2$, 
then the pushforward $f':=\mu_* f$ is the class of a minimal family so that $\Mk\cdot f=\Mk'\cdot f'=-2$.
\end{itemize}
\end{proposition}

\begin{proof}
\Mrefmclaim{a}
It follows from \LEM{pull} that $g$ is a rational family \Mst $\Mk\cdot g=-2$
and $\Mh\cdot g=\Mh'\cdot g' + 2$. 
Suppose that $f\in N(Y)$ is the class of a minimal family.
By \LEM{push} we have $\Mh\cdot f=\Mh'\cdot f' - \Mk\cdot f$. 
Since $-\Mk\cdot f\geq2$ by \LEM{kf}, we find that
\[
\Mh\cdot f=\Mh'\cdot f' - \Mk\cdot f\geq \Mh'\cdot f' + 2\geq \Mh'\cdot g' + 2=\Mh\cdot g,
\]
and thus $g$ is the class of a minimal family.

\Mrefmclaim{b}
We know from \LEM{push} that $f'$ is the class of a rational family \Mst
$\Mh\cdot f=\Mh'\cdot f'- \Mk\cdot f$.
It follows from \Mrefmclaim{a} that $\Mh\cdot f=\Mh\cdot g$ where $g=\mu^* g'$.
As a consequence of \LEM{pull} we have $\Mh\cdot g=\Mh'\cdot g' +2$
and thus
\[
\Mh\cdot f= \Mh'\cdot f' - \Mk\cdot f=\Mh'\cdot g' +2=\Mh\cdot g.
\]
Since $g'$ is minimal and $\Mk\cdot f\leq -2$ by \LEM{kf}, it follows that
$f'$ is minimal and $\Mk\cdot f=-2$.
\end{proof}

\begin{lemma}
\textbf{(pullback of lines)}
\label{lem:lines}
\\
Let $\Mrow{\mu}{(Y,\Mh)}{(Y',\Mh')}$ be an adjoint relation of type 1 \Mst $Y'\cong\MbbP^2$.
If $\sigma_*(\Me_j)=\Me_j$ for some $j>0$,
then
$S(Y,\Mh)=\Mset{ \Me_0-\Me_j}{ \sigma_*(\Me_j)=\Me_j,~ j>0}$.

\end{lemma}

\begin{proof}
Suppose that $f=\Me_0-\Me_j$ for some $1\leq j\leq r$ \Mst $\sigma_*(f)=f$.
We notice that $f':=\mu_*f=\Me_0$, $\Mk\cdot f=-2$ and
by \PRP{mp}.5 we have $S(Y',\Mh')=\{ f' \}$.
Suppose by contradiction that $g$ is the class of a minimal family
\Mst $\Mh\cdot g<\Mh\cdot f $ and let $g':=\mu_*g$. 
It follows from $\Mh'\cdot g'\geq \Mh'\cdot f'$ and \LEM{push}
that $\Mh\cdot g + \Mk\cdot g \geq \Mh\cdot f + \Mk\cdot f$.
Thus $ \Mk\cdot g \geq \Mh\cdot f-\Mh\cdot g + \Mk\cdot f \geq -1$.
We arrived at a contradiction with \LEM{kf}.
\end{proof}

\begin{remark}
\textbf{\textit{(complex minimal families)}}
\label{rmk:complex}
\\
Suppose that real structure acts as the identity
on the NS-lattices.
By \PRP{mp}, either $\Mk_\ell\cdot f=-2$ for some $f\in S(Y_\ell,h_\ell)$
or $Y_\ell\cong\MbbP^2$.  
The classification of classes of minimal families is in this case a consequence 
of \PRP{pp} and \LEM{lines} \Mresp.
Thus up to now we recovered the classification of complex minimal families in 
\citep[Theorem~46]{nls1}, since the real structure 
does not impose additional restrictions.
In the following subsection we take care of the complications
coming from $\sigma_*$ not being the identity. 
\Mend
\end{remark}

\subsection{}

\begin{remark}
\textbf{\textit{(key idea for final part of classification proof)}}
\label{rmk:intro}
\\
When an adjoint relation is realized by the blowup of the projective plane 
in complex conjugate points, then the canonical degree of the minimal family 
is not guaranteed to be -2, so that we cannot apply \PRP{pp}. The key idea
is that the pullback of conics through four points along the adjoint relation
is a candidate for a minimal family and we show that this 
candidate can only be beaten by the pullback of lines.  
The proof is somewhat technical and in order 
to appreciate the key idea, the reader is encouraged to try to prove 
that the classes of the families in \EXM{chain} are indeed minimal.
Notice that we do not assume that the centers of the blowup 
are in general position and minimal families are a priori not necessarily complete.
\Mend 
\end{remark}

\begin{assumption} 
\label{asm:1}
We consider an adjoint chain of type 1 and coordinates for $\Mh_i$, $\Mk_i$ and $f_i$ 
as stated in \LEM{coord}.\Mrefmclaim{a} for $0\leq i\leq \ell$.
Moreover, we suppose that $f_0\in S(Y_0,h_0)$. 
If $\Mrnk(N(Y_0))\geq 5$, then
$g_0:=2\Me_0-\Me_1-\Me_2-\Me_3-\Me_4$ in $N(Y_0)$;
If $1<\Mrnk(N(Y_0))\leq 4$, then $g_0:=2\Me_0-\Me_1-\ldots -\Me_r$, where $r:=\Mrnk(N(Y_0))-1$.
We define $g_i:=(\mu_{i-1}\circ\ldots\circ\mu_0)_*(g_0)$ for $0< i\leq \ell$.
\Mend
\end{assumption}

\begin{lemma}
\textbf{(irreducibility of $g_0$)} 
\label{lem:g0}
We use \ASM{1}.
\\
If $\Mh_0\cdot g_0< \Mh_0\cdot \Me_0$, 
then $g_0$ is the class of a rational family.
\end{lemma}

\begin{proof}
We have $\Mh_0\cdot g_0=\Mh_0\cdot\Me_0 + \Mh_0\cdot(\Me_0-\Me_1-\ldots -\Me_r) < \Mh_0\cdot\Me_0$ for $1\leq r\leq 4$, 
so that $\Mh_0\cdot(\Me_0-\Me_1-\ldots-\Me_r)<0$. Thus $h^0(\Me_0-\ldots-\Me_r)=0$,
since $\Mh_0$ is nef by definition and thus nonnegative against effective classes.
We established that $g_0$ has no fixed components. 
The adjoint chain is of type 1 and thus $\sigma_*(g_0)=g_0$.
The rationality of the family with class $g_0$ follows from being the pullback of conics in the plane
that pass through $r$ base points.
\end{proof}

\begin{lemma}
\textbf{(bounding $\Me_0\cdot f_0$ for small rank)}
\label{lem:r4}
We use \ASM{1}.
\\
If 
$\Mh_s\cdot f_s\leq \Mh_s\cdot g_s$
and 
$\Mrnk(N(Y_s))\leq 4$ for some $0\leq s\leq\ell$,
then $\Me_0\cdot f_0 \leq 2$.
\end{lemma}

\begin{proof}
If $\Mrnk(N(Y_s))=1$, then $f_s=\beta_0\Me_0$ and $g_s=2\Me_0$ by \LEM{coord} so that $\Me_0\cdot f_0 \leq 2$.

If $2\leq \Mrnk(N(Y_s))\leq 4$ and $\sigma_*(\Me_1)=\Me_1$, then 
$S(Y_\ell,\Mh_\ell)\subseteq \Mset{ \Me_0-\Me_i}{i>0}\cup \{\Me_0\}$ by \PRP{mp}.
Thus $S(Y_s,\Mh_s)\subseteq \Mset{ \Me_0-\Me_i}{i>0}$ by \LEM{lines} and \PRP{pp}. 
Recall that by assumption $f_0\in S(Y_0,\Mh_0)$.
Therefore, by \PRP{pp}, one has $f_s\in S(Y_s,\Mh_s)$, so that in this case the lemma holds.

Suppose that $\Mrnk(N(Y_s))=3$ and $\sigma_*(\Me_1)=\Me_2$.
We use the notation in \LEM{coord} and let
$(\alpha_i)_i$ be the coefficients of $\Mh_0$.
Let $\alpha_0':=\alpha_0-3s$ and $\alpha_j':=\alpha_j-s$ for $0< j\leq r$ and notice that $\alpha_1'=\alpha_2'$
so that $\Mh_s=\alpha_0'\Me_0-\alpha_1'\Me_1-\alpha_1'\Me_2$ and $\sigma_*(\Mh_s)=\Mh_s$.
By \LEM{coord} and $\sigma_*(f_s)=f_s$ one has $g_s=2\Me_0-\Me_1-\Me_2$ and $\beta_1=\beta_2$. 
Since $\Mh_s\cdot f_s\leq \Mh_s\cdot g_s$, we have 
$
0\leq \alpha_0'\beta_0-2\alpha_1'\beta_1\leq 2\alpha_0'-2\alpha_1'.
$
Notice that $f_s$ is the class of a rational family 
and therefore does not have fixed components. 
In particular, $f_s\cdot (\Me_0-\Me_1-\Me_2)\geq 0$ so that 
$2\beta_1 \leq \beta_0$.
It follows that $\beta_0(\alpha_0'-\alpha_1')\leq 2(\alpha_0'-\alpha_1')$
and thus $\beta_0\leq 2$ where $\beta_0=\Me_0\cdot f_0$.

Finally, suppose that $\Mrnk(N(Y_s))=4$, $\sigma_*(\Me_1)=\Me_2$ and $\sigma_*(\Me_3)=\Me_3$.
Let $z_0=\Me_0-\Me_3$ in $N(Y_0)$ be the class of a rational family 
so that $\sigma_*(z_0)=z_0$ and $\Mk_0\cdot z_0=-2$. 
We define $z_i:=(\mu_{i-1}\circ\ldots\circ\mu_0)_*(z_0)$ for $0< i\leq \ell$,
in accordance with the notation of \ASM{1}.
We may assume that $\Mrnk(N(Y_s))>\Mrnk(N(Y_\ell))$,
otherwise it follows from \PRP{mp} and \PRP{pp} that $\Me_0\cdot f_0\leq 2$.
Let $t$ be the smallest value \Mst $s< t\leq \ell$ and $\Mrnk(N(Y_t))\leq 3$.
Since $f_0\in S(Y_0,\Mh_0)$ by assumption, we have $\Mh_0\cdot f_0 \leq \Mh_0\cdot z_0$
and it follows from \LEM{kf} that $\Mk_i\cdot f_i \leq \Mk_i\cdot z_i=-2$ for $0\leq i < t$.
We apply \LEM{push} and find that 
$\Mh_i\cdot f_i\leq \Mh_i\cdot z_i$ for all $0\leq i\leq t$.
By \LEM{coord} we find that $z_t=\Me_0$ and thus 
\[
\Mh_t\cdot f_t \leq \Mh_t\cdot\Me_0 \leq \Mh_t\cdot g_t = \Mh_t\cdot\Me_0+\Mh_t\cdot (g_t-\Me_0).
\]
This concludes the proof of this lemma, since we already treated the cases where
$\Mrnk(N(Y_t))\leq 3$ and $\Mh_t\cdot f_t \leq \Mh_t\cdot g_t$.
\end{proof}

\begin{lemma}
\textbf{(bounding $\Me_0\cdot f_0$)}
\label{lem:ef}
We use \ASM{1}. 
\\
If 
$\Mh_0\cdot f_0 \leq  \Mh_0\cdot g_0$, 
then either $\Me_0\cdot f_0 \leq 2$ or $f_\ell\in S(Y_\ell,\Mh_\ell)$ \Mst $\Mk_\ell\cdot f_\ell=-2$.
\end{lemma}

\begin{proof}
If $\Mrnk(N(Y_0))\leq 4$, then by \LEM{r4} one has $\Me_0\cdot f_0 \leq 2$.
We assume in the remainder of this proof that $\Mrnk(N(Y_0))>4$
so that $\Mk_0\cdot g_0=-2$.
By \LEM{kf} we have that $\Mk_0\cdot f_0\leq \Mk_0\cdot g_0$.    

If $\Mk_i\cdot f_i\leq \Mk_i\cdot g_i$ for all $0\leq i\leq \ell$,
then by \LEM{push} we have $\Mh_i\cdot f_i \leq  \Mh_i\cdot g_i$
for all $0\leq i\leq \ell$.
If $\Mrnk(N(Y_\ell))\leq 4$, then by \LEM{r4} one has $\Me_0\cdot f_0 \leq 2$.
If $\Mrnk(N(Y_\ell))> 4$, then 
it follows from \PRP{mp}.[1,2] and \PRP{pp} that
$f_\ell\in S(Y_\ell,\Mh_\ell)$ \Mst $\Mk_\ell\cdot f_\ell=-2$.

Now, suppose that $0< s\leq \ell$ 
is the smallest number \Mst $\Mk_s\cdot f_s> \Mk_s\cdot g_s$.
It follows from \LEM{coord} that
$g_s\in \{~2\Me_0-\Me_1-\Me_2-\Me_3,~2\Me_0-\Me_1-\Me_2,~2\Me_0-\Me_1,~2\Me_0~\}$
and thus $\Mrnk(N(Y_s))\leq 4$. 
By \LEM{push} we have $\Mh_s\cdot f_s \leq  \Mh_s\cdot g_s$
and thus by \LEM{r4} one has $\Me_0\cdot f_0 \leq 2$.
This concludes the proof of this lemma, since we considered all 
possible scenarios.
\end{proof}

\begin{lemma}
\textbf{(bounding $(\Ml_0+\Ml_1)\cdot f_0$)}
\label{lem:type2}
\\
We consider an adjoint chain of type 2 and coordinates for $\Mh_i$, $\Mk_i$ and $f_i$ 
as stated in \LEM{coord}.\Mrefmclaim{b} for $0\leq i\leq \ell$.
If $f_0\in S(Y_0,\Mh_0)$, then either 
$(\Ml_0+\Ml_1)\cdot f_0\leq 2$
or $f_\ell\in S(Y_\ell,\Mh_\ell)$ \Mst $\Mk_\ell\cdot f_\ell=-2$.
\end{lemma}

\begin{proof}
If $\Mrnk(N(Y_0))< 4$, then it follows from \PRP{mp}.[3,4,5] that 
$f_0\in\{~ \Ml_0,~ \Ml_1,~ \Ml_0+\Ml_1,~ \Ml_0+\Ml_1-\Mp_1 ~\}$ so that this lemma holds.
If $\Mrnk(N(Y_\ell))\geq 4$, then
we know from \PRP{mp}.[1,2] and \PRP{pp} that
$f_\ell\in S(Y_\ell,\Mh_\ell)$ \Mst $\Mk_\ell\cdot f_\ell=-2$.

We assume that $\Mrnk(N(Y_0))\geq 4$ and $\Mrnk(N(Y_\ell))<4$
in the remainder of this proof.
We set $g_0:=\Ml_0+\Ml_1-\Mp_1-\Mp_2$ and define 
$g_i:=(\mu_{i-1}\circ\ldots\circ\mu_0)_*(g_0)$ for $0< i\leq \ell$.
We proceed similarly as in \LEM{g0} and \LEM{ef}.

If the linear series with class $g_0$ has fixed components, 
then $h^0(\Ml_j-\Mp_1-\Mp_2)=1$ where \Mwlog $j=1$.
Thus $g_0=\Ml_0+(\Ml_1-\Mp_1-\Mp_2)$ with $\sigma_*(\Ml_0)=\Ml_0$ 
\Mst $\Ml_0 \in S(Y_\ell,\Mh_\ell)$ with $\Mk_\ell\cdot \Ml_0=-2$.
It follows from \PRP{pp} that in this case $f_\ell\in S(Y_\ell,\Mh_\ell)$ \Mst $\Mk_\ell\cdot f_\ell=-2$
as asserted.
So we may assume \Mwlog that $g_0$ has no fixed components.
By the arithmetic genus formula $p_a(g_0)=\frac{1}{2}(g_0^2+\Mk_0\cdot g_0)+1=0$
and thus $g_0$ is the class of a rational family. 

Let $0< s\leq \ell$ be the minimal value so that $\Mrnk(N(Y_s))< 4$.
By assumption we have $\Mh_0\cdot f_0 \leq \Mh_0\cdot g_0$
and it follows from \LEM{kf} that $\Mk_0\cdot f_0 \leq \Mk_0\cdot g_0=-2$.
It follows from \LEM{push} that $\Mh_i\cdot f_i \leq \Mh_i\cdot g_i$
for all $0\leq i\leq s$. We already established 
that $f_s\in\{~ \Ml_0,~ \Ml_1,~ \Ml_0+\Ml_1,~ \Ml_0+\Ml_1-\Mp_1 ~\}$ 
and therefore by \LEM{coord} we find that $(\Ml_0+\Ml_1)\cdot f_0\leq 2$,
which concludes the proof of this lemma.
\end{proof}

\begin{lemma}
\textbf{(pullback of the class $\Me_0$)}
\label{lem:e0}
\\
Let $\Mrow{\mu}{(Y,\Mh)}{(Y',\Mh')}$ be an adjoint relation.
If $S(Y',\Mh')=\{\Me_0\}$, 
then for all $f\in S(Y,\Mh)$ either $f=\Me_0$ or $\Mk\cdot f=-2$. 
\end{lemma}

\begin{proof}
Suppose that $f\in S(Y,\Mh)$ and $f'=\mu_*f$.
It follows from \LEM{push} that 
$f'$ is the class of a rational family \Mst
$\Mh'\cdot f'=\Mh\cdot f+\Mk\cdot f$.
We know from \LEM{kf} that $\Mk\cdot f\leq -2$.
Notice that $\Mk\cdot f\geq -3$, otherwise $f'\in S(Y',\Mh')$ and $f'\neq \Me_0$, 
contradicting the assumption.   
If $\Mk\cdot f=-3$, then $\Me_0\cdot f=1$ otherwise $f'\in S(Y',\Mh')$ and $f'\neq \Me_0$.
This concludes the proof, since if $\Mk\cdot f=-3$ and $\Me_0\cdot f=1$,
then $f=\Me_0$.
\end{proof}

\begin{theorem}
\textbf{(classification of real minimal families)}
\label{thm:cls}
\\
Suppose that an algebraic surface $X\subset \MbbP^n$ contains a rational curve through a general point.
Let $(Y_0,\Mh_0)$ be its associated ruled pair with adjoint chain
\[
\Marrow{(Y_0,\Mh_0)}{\mu_0}{(Y_1,\Mh_1)}\Marrow{}{\mu_1}{}\ldots\Marrow{}{\mu_{\ell-1}}{(Y_\ell,\Mh_\ell)},
\]
and let $f_0\in S(Y_0,\Mh_0)$ be the class of a real minimal family.
If $X$ is either $\MbbR$-rational or $\Mk_\ell\cdot z_\ell=-2$ for some $z_\ell\in S(Y_\ell,\Mh_\ell)$,
then one of the following cases holds: 
\begin{enumerate}[topsep=0pt, label=(\roman*)]
\item 
$f_0$ is the pullback of $f_\ell\in S(Y_\ell,\Mh_\ell)$ as specified in \PRP{mp} 
with the additional property that $\Mk_\ell\cdot f_\ell=-2$.

\item
the adjoint chain is of type 1 so that, up to permutation of $(\Me_j)_{j>0}$,
\[
f_0\in\{~\Me_0,~\Me_0-\Me_1,~2\Me_0-\Me_1-\Me_2-\Me_3-\Me_4~\}.
\]

\item
the adjoint chain is of type 2 so that, up to permutation of $(\Mp_j)_{j>0}$, 
\[
f_0\in\{~\Ml_0+\Ml_1,~\Ml_0+\Ml_1-\Mp_1,~ \Ml_0+\Ml_1-\Mp_1-\Mp_2 ~\}.
\]

\end{enumerate}
\end{theorem}

\begin{proof} 
If $\Mk_\ell\cdot z_\ell=-2$ for some $z_\ell\in S(Y_\ell,\Mh_\ell)$,
then we are in case (i) by \PRP{pp}.
In the remainder of the proof we assume that $X$ is $\MbbR$-rational.
We make a case distinction on $S(Y_\ell,h_\ell)$ in \PRP{mp}.

Suppose that we are in case (1), (2), (6) or (7) of \PRP{mp}.
By \LEM{psi0}, there exists $z_\ell\in S(Y_\ell,\Mh_\ell)$ \Mst $\Mk_\ell\cdot z_\ell=-2$
and thus we are in case (i) by \PRP{pp}.

Now suppose that $S(Y_\ell,\Mh_\ell)=\{\Me_0\}$ in case (3) or (5) of \PRP{mp}.
By \LEM{basis} the adjoint chain is of type 1.
It follows from \LEM{e0} that either $f_0=\Me_0$ as in case (ii) or $\Mk_0\cdot f_0=-2$.
Now suppose that $\Mk_0\cdot f_0=-2$ so that $\Mh_0\cdot f_0\leq\Mh_0\cdot \Me_0$
and let $g_0$ be defined as in \ASM{1}.
We have $\Mh_0\cdot f_0 \leq \Mh_0\cdot g_0$, 
otherwise $\Mh_0\cdot g_0 < \Mh_0\cdot\Me_0$ so that by \LEM{g0} 
there exists a family with class $g_0$
that is of lower degree than the family with class $f_0$.
Thus it follows from \LEM{ef} that $\Me_0\cdot f_0\leq 2$.
Since $\Mk_0\cdot f_0=-2$ and $\Me_0\cdot f_0\leq 2$,
we conclude that we are again in case (ii).

Finally, suppose that $S(Y_\ell,\Mh_\ell)=\{\Ml_0+\Ml_1\}$ is as in case (4) of \PRP{mp}.
By \LEM{basis} the adjoint chain must be of type 2. 
It follows from \LEM{type2} that we are in 
case (iii) and this concludes the proof of the theorem. 
\end{proof}

\begin{proof}[Proof of \COR{cls}]
~\\
\Mrefmclaim{b} Direct consequence of Riemann-Roch theorem
and Serre duality applied to the classes listed in \THM{cls}.

\Mrefmclaim{a}
Let $F\subset X\times B$ denote the minimal family of dimension of at least 3.
It follows from \THM{cls} and \Mrefmclaim{b} that $[F]=\Ml_0+\Ml_2$ with $\sigma_*(\Ml_0)=\Ml_1$.
Thus the linear normalization $Y_0$ of $X$ is isomorphic to $\MbbP^1\times\MbbP^1$, 
where the real structure $\Mrow{\sigma}{Y_0}{Y_0}$ flips the $\MbbP^1$-factors. 
The class of hyperplane sections is $\Mh_0=\alpha\Ml_0+\alpha\Ml_1 \in N(X)$ for some $\alpha\in\MbbZ_{>0}$.
A projective sphere has pair $(\MbbP^1\times\MbbP^1, \Ml_0+\Ml_1)$
and thus $Y_0$ and the sphere are isomorphic.

\Mrefmclaim{c}
It follows from \RMK{complex}, that
a minimal family over the complex 
numbers is either specified by case (i) of \THM{cls},
has class $\Me_0$ or is the pullback of lines through a point 
as characterized by \LEM{lines} with $\sigma_*$ the identity. 
We can conclude the proof, since
minimal families in all the remaining cases of \THM{cls} are
complete.
\end{proof}

The following corollary extends the
classification of multiple conical surfaces in \citep[Theorem~8 and Theorem~10]{sch6},
while also incorporating the real structure.

\begin{corollary}
\textbf{(classification real families of conics)}
\label{cor:conic}
\\
Suppose that $X\subset\MbbP^n$ is a real algebraic surface with ruled pair $(Y,\Mh)$.
If $X$ contains $\lambda\geq 1$ conics through
a general point, then either one of the following holds:
\begin{enumerate}[topsep=0pt, itemsep=0pt]

\item The pair $(Y,\Mh)$ is a minimal ruled pair and
either $X$ is geometrically ruled with $\deg X\leq n-1\leq 3$, or
the conics form minimal families with classes characterized by \PRP{mp}.[2,3,4,5].

\item $\lambda=1$ and there exists an adjoint relation 
$\Mrow{\mu}{(Y,\Mh)}{(Y',\Mh')}$ \Mst $(Y',\Mh')$ is a minimal ruled pair.
The conics form a minimal family whose class is
the pullback along $\mu$ of the class in \PRP{mp}.6.

\end{enumerate}
\end{corollary}

\begin{proof}
Suppose that $(Y_\ell,\Mh_\ell)$ is the minimal ruled pair of $X$.
It follows from \LEM{push} and \LEM{kf} that 
the degrees of minimal families on adjoint surfaces differ by at least two and
thus
$0\leq \Mh_\ell\cdot f_\ell\leq 2$ for $f_\ell\in S(Y_\ell,\Mh_\ell)$.
If $\Mh_\ell\cdot f_\ell= 0$, then we are in case (2) of the corollary by \PRP{pp}. 
If $\Mh_\ell\cdot f_\ell>0$, then \LEM{push} and \LEM{kf} ensure that $\ell=0$ so that 
$(Y,\Mh)$ is a minimal ruled pair.
\PRP{mp}.1 cannot hold, since the map associated to $\Mh_\ell=-\Mk_\ell$ is not birational in this case,
by \LEM{psi2a} and \LEM{psi4}.
Thus if $\Mh_\ell\cdot f_\ell=2$, then
the classes of conics on $X$ are classes of minimal families as 
characterized by \PRP{mp}.[2,3,4,5].

In the remainder of the proof we suppose that $\Mh_\ell\cdot f_\ell=1$ and $\deg X>2$.  
Thus $X$ is covered by both lines and conics so that the arithmetic genus of $X$ is zero.
It follows from \PRP{mp}.7, \citep[Proposition~2.c]{nls-f6} and \citep[Proposition~IV.18]{bea1} 
that $Y_\ell$ is a Hirzebruch surface \Mst $N(Y_\ell)=\Mmod{t,f}$ with 
$t^2=:r \geq 0$, 
$f^2=0$, 
$t\cdot f=1$, 
$\sigma_*=$ id,
$\Mk_\ell=-2t+(r-2)f$,
$h^0(t-rf)>0$
and
$2\Mh_\ell+\Mk_\ell=a f$ for some $a>0$. 
Suppose that $c:=\alpha t+ \beta f$ is the class of a conic for some $\alpha,\beta\in\MbbZ$.
It follows from $\Mh_\ell\cdot c=2$, $c\cdot f=\alpha$ and $c\cdot (t-rf)=\beta$ 
that $\alpha\geq 0$ and $\beta=2-\frac{1}{2}\alpha(r+a+2)\geq 0$.
Since $c\neq \beta f$ we find that $c=t$ and $(r,a)\in\{(1,1),(0,2)\}$.
If $(r,a)=(1,1)$, then $\Mh_\ell=t+f$, $h^0(\Mh_\ell)=5$ and $h^0(c)=3$ by Riemann-Roch theorem 
and Kodaira vanishing theorem so that $X$
is a cubic geometrically ruled surface whose linear normalization is in $\MbbP^4$.
If $(r,a)=(0,2)$, then $\Mh_\ell=t+2f$ with $h^0(c)=2$ by Riemann-Roch theorem and Kodaira vanishing theorem, 
contradicting that $X$ contains two conics through a general point.

We considered all possible values of $\Mh_\ell\cdot f_\ell$ and therefore concluded the proof.
\end{proof}

\section{Computing minimal families}
\label{sec:alg}

In this section we reformulate the results of \SEC{min} into an 
algorithmic form.
The following diagram depicts a birational map 
$\Mdashrow{\McalH}{\MbbP^2}{X\subset\MbbP^n}$ 
and the resolution of its baselocus
\begin{equation}
\label{eqn:res}
\begin{array}{r@{}c@{}l}
        &           Z                  &     \\
        &\tau_1\swarrow~~\searrow\tau_2 &     \\
\MbbP^2 & \Mdasharrow{}{\McalH}{}      & X          
\end{array}
\end{equation}
If $Z$ is not biregular to the smooth model~$Y$ of~$X$, then it is more 
convenient to compute with $N(Z)$ instead of $N(X)$.
We denote by $\Mhh,\Mkk\in N(Z)$ the class of hyperplane sections and 
the canonical class \Mresp.
Let 
\[
S_*(Z,\Mhh)=\Mset{ f\in S(Z,\Mhh) }{ f\cdot e=0 \text{ for all } e\in\McalE},
\]
where $\McalE=\Mset{ e\in N(Z)}{ e^2=\Mkk\cdot e=-1,~\Mhh\cdot e=0 }$.
Notice that $N(X)\cong N(Z)/\McalE$.

\newpage
\begin{example}
\textbf{(sphere)}
\label{exm:sphere}
\\
Let $X:=\Mset{ x\in \MbbP^3}{ -x_0^2+x_1^2+x_2^2+x_3^2=0 }$ be the projective 
closure of the unit-sphere. 
The projection of $X$ with center $(1:0:0:1)$ is called
a \Mdef{stereographic projection} and defined as
\[
\Mdasharrow{X}{}{\MbbP^2},~ x \mapsto (x_0-x_3:x_1:x_2).
\] 
The inverse of this stereographic projection defines a parametrization of $X$:
\[
\Mdashrow{\McalH}{\MbbP^2}{X},~ x\mapsto (x_0^2+x_1^2+x_2^2+x_3^2:2x_0x_1:2x_0x_2:-x_0^2+x_1^2+x_2^2+x_3^2).  
\]
We find with \citep[Algorithm~1]{nls-bp} that the base locus of $\McalH$
consists of two points $(\pm\Mi:1:0)$ in $\MbbP^2$. We denote the blowup of
$\MbbP^2$ in these two points by $\Mrow{\tau_1}{Z}{\MbbP^2}$
so that $N(Z)=\Mmod{\Me_0,\Me_1,\Me_2}$, 
$\sigma_*(\Me_0)=\Me_0$, 
$\sigma_*(\Me_1)=\Me_2$,
$\Mhh=2\Me_0-\Me_1-\Me_2$ and $\Mkk=-3\Me_0+\Me_1+\Me_2$.
Notice that $Z$ is 
not the smooth model of $X$, since the class in $\McalE=\{~\Me_0-\Me_1-\Me_2~\}$ is orthogonal to $\Mhh$.
The corresponding line through the basepoints 
is an exceptional curve and contracted by $\Mrow{\tau_2}{Z}{X}$ to a smooth point on $X\cong\MbbP^1\times\MbbP^1$.
The classes of the families of complex lines in $X$ are $\Me_0-\Me_1$ and $\Me_0-\Me_2$.
The class of the family of conics through the center of projection $(1:0:0:1)$ is $\Me_0$.
The class of the family of all conics on $X$ is the class of hyperplane sections 
$2\Me_0-\Me_1-\Me_2$. 
A minimal family must have conics as members and thus $S_*(Z,\Mhh)=\{~2\Me_0-\Me_1-\Me_2~\}$.
\Mend
\end{example}

If $c\in N(X)$ is a class, then $c=\MmfM(c)+\MmfF(c)$ is its decomposition into
its 
\Mdef{moving component} $\MmfM(c)$ and 
\Mdef{fixed component} $\MmfF(c)$
so that $h^0(\MmfM(c))>1$, $h^0(\MmfF(c))=1$ and $\MmfF(\MmfM(c))=0$.

If $h^0(\Mhh+\Mkk)>1$, then a \Mdef{pseudo adjoint relation} is defined as
\[
\Mrow{\lambda}{(Z,\Mhh)}{(Z',\Mhh'):=\bigl(~\lambda(Z),~ \MmfM\bigl(\lambda_*(\Mhh+\Mkk)\bigr)  ~\bigr)},
\]
where
$\Mrow{\lambda}{Z}{Z'}$ contracts exceptional curves $E\subset Z$ with the property that $(\Mhh+\Mkk)\cdot [E]= 0$
and $[E]\cdot[C]=0$ for all exceptional curves $C$ \Mst $\Mhh\cdot [C]=0$.
Following \SEC{chain} we obtain a \Mdef{pseudo adjoint chain}
\[
\Marrow{(Z_0,\Mhh_0)}{\lambda_0}{(Z_1,\Mhh_1)}\Marrow{}{\lambda_1}{}\ldots\Marrow{}{\lambda_{\ell-1}}{(Z_\ell,\Mhh_\ell)}.
\]
We call $(Z_\ell,\Mhh_\ell)$ a \Mdef{pseudo minimal ruled pair}.

The following lemma relates the pseudo adjoint chain to the adjoint chain.

\begin{lemma}
\textbf{(pseudo adjoint relation and minimal families)}
\label{lem:par}
\\
Suppose that $\Mrow{\lambda_0}{(Z_0,\Mhh_0)}{(Z_1,\Mhh_1)}$ is a psuedo adjoint relation.
We consider the following commutative diagram
\begin{equation}
\label{eqn:diagram}
\begin{array}{ccc}
(Z_0,\Mhh_0)       & \Marrow{}{\lambda_0}{} & (Z_1,\Mhh_1)       \\
\downarrow\gamma_0 &                        &  \downarrow\gamma_1\\
(Y_0,\Mh_0)        & \Marrow{}{\mu_0}{}     & (Y_1,\Mh_1)        \\
\end{array}
\end{equation}
where
$\mu_0$ is an adjoint relation, $\Mh_0=\gamma_{0*}\Mhh_0$
and
$\Mrow{\gamma_0}{Z_0}{Y_0}$ contracts exceptional curves $C\subset Z_0$ \Mst $\Mhh_0\cdot [C]= 0$
and $\gamma_1=\mu_0\circ\gamma_0$.

Moreover, we assume that $N(Z_0)$ has a type 1 basis \Mst 
$\Me_0$ is the pullback of lines in $\MbbP^2$
and $(\Me_i)_{i>0}$ are the pullback of exceptional curves, via a birational morphism $\Marrow{Z_0}{}{\MbbP^2}$.
\begin{itemize}[topsep=0mm, itemsep=0mm]

\Mmclaim{a}
We have $\Mhh_0=\gamma_0^*\Mh_1$,  $\Mhh_1=\gamma_1^*\Mh_1$ and
$\Mrow{\gamma_1}{Z_1}{Y_1}$ contracts exactly exceptional curves $C'\subset Z_1$ \Mst $\Mhh_1\cdot [C']= 0$.

Moreover, $N(Z_1)$ has a type 1 basis \Mst 
$\Me_0$ is the pullback of lines in $\MbbP^2$
and $(\Me_i)_{i>0}$ are the pullback of exceptional curves, via a birational morphism $\Marrow{Z_1}{}{\MbbP^2}$.

There is a natural inclusion $\Mhookrow{\iota}{N(Z_1)}{N(Z_0)}$ 
that preserves the generators of the type 1 basis.

\Mmclaim{b}
We have
$S_*(Z_i,\Mhh_i)=\Mset{\gamma_i^*f}{f\in S(Y_i,\Mh_i)}$ for $0\leq i\leq 1$.

\Mmclaim{c}
If there exists $f\in S_*(Z_1,\Mhh_1)$ \Mst $\Mkk_1\cdot f=-2$, then
\[
S_*(Z_0,\Mhh_0)=\Mset{ \lambda_0^*f }{f\in S_*(Z_1,\Mhh_1)\text{ and }\Mkk_1\cdot f=-2 }. 
\]

\Mmclaim{d}
If $f\in S(Y_1,\Mh_1)\cap\Psi_j$, then $\gamma_0^*f \in S_*(Z_1,\Mhh_1)\cap\Psi_j$, for all $j\in\{0,1,2,4\}$. 

\end{itemize}
\end{lemma}

\newpage
\begin{proof}
Let $c_1,\ldots,c_s\in N(Z_0)$ denote the classes of exceptional curves that are contracted by $\gamma_0$.

\Mclaim{1} $\Mhh_0=\gamma_0^*\Mh_0$ and $\Mhh_1=\gamma_1^*\Mh_1$.
\\
Since $\gamma_0$ contracts exceptional curves that are orthogonal
to $\Mhh_0$ it follows that $\Mhh_0=\gamma_0^*\Mh_0$.
We recall from \SEC{chain} that $(Y_1,\Mh_1)$ is a ruled pair
so that $\Mh_0+\Mk_0$ has no fixed components.
Consequently, $\gamma_0^*(\Mh_0+\Mk_0)$ has no fixed components as 
$\gamma_0^*(\Mh_0+\Mk_0)\cdot c_j=0$ for all $1\leq j\leq s$.
By \citep[(1.41)]{deb1} we have
\begin{equation}
\label{eqn:k}
\Mkk_0=\gamma_0^*\Mk_0+c_1+\ldots+c_s. 
\end{equation}
Thus
$\Mhh_0+\Mkk_0=\gamma_0^*(\Mh_0+\Mk_0)+c_1+\ldots+c_s$
so that
$\MmfF(\lambda_{0*}(\Mhh_0+\Mkk_0))=\lambda_{0*}(c_1+\ldots+c_s)$,
since 
$(\Mhh_0+\Mkk_0)\cdot c_j<0$ for $1\leq j\leq s$.
It follows that
$\Mhh_1=\lambda_{0*}\gamma_0^*(\Mh_0+\Mk_0)=\lambda_{0*}(\mu_0\circ\gamma_0)^*\Mh_1=\gamma_1^*\Mh_1$,
as claimed.

\Mclaim{2} {\it If $e\in N(Z_0)$ is an exceptional curve contracted by $\lambda_0$,
then $\gamma_{0*}e\in N(Y_0)$ is an exceptional curve contracted by $\mu_0$.
Moreover, exceptional curves contracted by $\lambda_0$ are disjoint
from the curves that are contracted by $\gamma_0$.
}
\\
Since $(\Mhh_0+\Mkk_0)\cdot e=0$
it follows that $\Mhh_0\cdot e=\gamma_0^*\Mh_0\cdot e=\Mh_0\cdot\gamma_{0*}e=1$.
Moreover,
$\Mkk_0\cdot e=(\gamma_0^*\Mk_0+c_1+\ldots+c_s)\cdot e=-1$
and thus $\Mk_0\cdot \gamma_{0*}e=-1-c_1\cdot e-\ldots-c_s\cdot e$
so that $(\Mh_0+\Mk_0)\cdot \gamma_{0*}e=-c_1\cdot e-\ldots-c_s\cdot e$. 
Since $\Mh_0+\Mk_0$ has no fixed components, 
it follows that $-c_1\cdot e-\ldots-c_s\cdot e\geq 0$ and thus $c_1\cdot e=\ldots=c_s\cdot e=0$.
The claim now follows from $(\gamma_{0*}e)^2=\Mk_0\cdot \gamma_{0*}e=-1$.

\Mclaim{3} {\it If $b\in N(Y_0)$ is the class of an exceptional curve 
contracted by $\mu_0$, then $\gamma^*_0b\in N(Z_0)$ 
is the class of an exceptional curve contracted by $\lambda_0$.
}
\\
Let $\gamma_0^*b=b'+m_1c_1+\ldots+m_sc_s$, where $b'$ is the strict transform of $b$
and $m_i\in \MbbZ_{\geq 0}$. It follows from the projection formula and \EQN{k} that
$
\Mk_0\cdot b
=
\gamma_{0*}\Mkk_0 \cdot b
=
\Mkk_0 \cdot \gamma_0^* b
=
\Mkk_0 \cdot b' - m_0 - \ldots - m_s
=
-1
$ and $\Mkk_0 \cdot b'=\gamma_0^*\Mk_0 \cdot b'=\Mk_0 \cdot \gamma_{0*}b'=\Mk_0\cdot b=-1$.
Thus we find that $m_0=\ldots =m_s=0$ \Mst $\gamma^*_0b=b'$.
By \EQN{k}, \Mrefclaim{1} and $(\Mh_0+\Mk_0)\cdot b=0$, we have $(\Mhh_0+\Mkk_0)\cdot \gamma^*_0b=0$
with $(\gamma^*_0b)^2=\Mkk_0\cdot \gamma^*_0b=-1$ so that this claim holds.

\Mclaim{4} {\it The exceptional curves contracted by $\lambda_0$ and $\mu_0$
are disjoint.}
\\
Since $\Mhh_0^2>0$ and $\Mhh_0\cdot(c_1+c_2)=0$, it 
follows from Hodge index theorem that $(c_1+c_2)^2<0$ so that $c_1\cdot c_2=0$
and thus $c_i\cdot c_j=0$ for $i\neq j$. Similar argument shows this claim for $\mu_0$.

\Mclaim{5} {\it $\gamma_1$ contracts exceptional curves $C'\subset Z_1$
\Mst $\Mhh_1\cdot [C']=0$.}
%
%
\\
Since $\gamma_1=\mu_0\circ\gamma_0$,
it follows from claims 2, 3 and 4 
and Castelnuovo's contraction criterion
that $\gamma_1$ contracts exactly the exceptional curves with classes $\lambda_{0*}c_j$
for $1\leq j\leq s$. 
We notice that 
$\lambda^*_0\Mhh_1=\gamma_0^*(\Mh_0+\Mk_0)$, since the class of an exceptional curve
that is contracted by $\lambda_0$ is orthogonal to $\gamma_0^*(\Mh_0+\Mk_0)$.
Thus $\Mhh_1\cdot\lambda_{0*}c_j=\lambda^*_0\Mhh_1\cdot c_j=0$ as claimed.

%

\Mrefmclaim{a}
By claims 2, 3 and 4 we find that $N(Z_0)=\Mmod{\Me_0,\Me_1,\ldots,\Me_r}$
and $N(Z_1)=\Mmod{\Me_0,\Me_1,\ldots,\Me_{r'}}$ where 
either $r=r'$ or $r'<r$ and the exceptional curves with
classes $\Me_{r'},\Me_{r'+1},\ldots,\Me_{r}$ are contracted by $\lambda_0$.
These choices for generators induce the inclusion $\iota$.
Claims 1 and 5 conclude the proof for assertion \Mrefmclaim{a}.

\Mrefmclaim{b} 
If $g\in S_*(Z_i,\Mhh_i)$ where $0\leq i\leq 1$, then $\Mhh_i\cdot \gamma_i^*f \geq \Mhh_i\cdot g$
for all $f\in S(Y_i,\Mh_i)$.
It follows from \Mrefmclaim{a} that $\Mh_i\cdot f\geq \Mh_i\cdot \gamma_{i*} g$,
since $\Mhh_i\cdot \gamma_i^*f=\Mh_i\cdot f$ and $\Mhh_i\cdot g=\gamma_i^*\Mh_i \cdot g$.  
As $f\in S(Y_i,\Mh_i)$ we find that $\Mh_i\cdot f\leq \Mh_i\cdot \gamma_{i*} g$
and thus $\Mhh_i\cdot \gamma_i^*f = \Mhh_i\cdot g$.
Therefore $\gamma_i^*f \in S_*(Z_i,\Mhh_i)$ and $\gamma_{i*} g\in S(Y_i,\Mh_i)$ so that 
assertion \Mrefmclaim{b} holds.

\Mrefmclaim{c}
This assertion follows from \Mrefmclaim{b} and \PRP{pp}.

\Mrefmclaim{d}
We apply \Mrefmclaim{a} and subsequently consider adjoint relations 
until $(Y_\ell,\Mh_\ell)$ is a minimal ruled pair. By assumption,
$f\in\Psi_j$ for $j\in\{0,1,2,4\}$ and thus by \PRP{mp} and \PRP{chain}.\Mrefmclaim{c} 
the surface $Y_\ell$ is a weak del Pezzo surface.
By \Mrefmclaim{c} we may assume \Mwlog that $(Y_1,\Mh_1)=(Y_\ell,\Mh_\ell)$.
By \LEM{psi} the set $\Psi_j$ 
is constructed with \citep[Algorithm~1]{nls-algo-fam},
as the set of classes with canonical degree $-2$ 
and self-intersection $j$ \Mst the rank of the NS-lattice is at most 9. 
By \Mrefmclaim{a} the lattice $N(Z_1)$ has a type 1 basis $\Mmod{\Me_0,\ldots,\Me_r}$.
It follows from \EQN{k} that $\gamma_0^*f\in \Mset{ g\in N(Z_1)}{ \Mkk_1\cdot g=-2 \text{ and } g^2=j }$.
Thus, if $r\leq 8$, then $\gamma_0^*f\in\Psi_j$ by construction.
If $r>8$, then there must exist $e_t$ for $1\leq t \leq r$ \Mst $\gamma_0^*f\cdot e_t=0$,
since $f\in N(Y_1)$ and the rank of $N(Y_1)$ is at most 9.
It follows, after permutation of the generators, that $N(Z_1)=\Mmod{\Me_0,\ldots,\Me_8}\oplus\Mmod{\Me_9,\ldots,\Me_r}$
and $\gamma_0^*f\in \Mmod{\Me_0,\ldots,\Me_8}$ so that $\gamma_0^*f\in\Psi_j$.
We conclude from \Mrefmclaim{b} that $\gamma_0^*f\in S_*(Z_1,\Mhh_1)\cap \Psi_j$.
\end{proof}

\newpage
\begin{example}
\textbf{(pseudo adjoint chain)}
\label{exm:par}
\\
We consider a pseudo adjoint chain in a diagram following \LEM{par}, 
\begin{equation*}
\begin{array}{ccccccc}
(Z_0,\Mhh_0)       & \Marrow{}{\lambda_0}{} & (Z_1,\Mhh_1)        & \Marrow{}{\lambda_1}{}& (Z_2,\Mhh_2)     \\           
\downarrow\gamma_0 &                        &  \downarrow\gamma_1 &                       &  \downarrow\gamma_2 \\
(Y_0,\Mh_0)        & \Marrow{}{\mu_0}{}     & (Y_1,\Mh_1)         & \Marrow{}{\mu_1}{}    & (Y_2,\Mh_2)      \\
\end{array}
\end{equation*}
\Mst 
\[
\begin{array}{r@{\,=\,}r@{\,}r@{\,}r@{\,}r@{\,}r@{\,}r@{\,}r@{\,}r@{\,}r@{\,}}
\Mhh_0 & 10\Me_0 &-& 5\Me_1 &-& 5\Me_2 &-& 2\Me_3 &-& 2\Me_4, \\
\Mhh_1 &  6\Me_0 &-& 3\Me_1 &-& 3\Me_2 &-&  \Me_3 &-&  \Me_4, \\
\Mhh_2 &   2\Me_0 &-& \Me_1 &-&  \Me_2 &,~ &        & &        \\
\Mkk_0 & -3\Me_0 &+& \Me_1 &+& \Me_2 &+& \Me_3 &+& \Me_4, \\
\Mkk_1 & -3\Me_0 &+& \Me_1 &+& \Me_2 &+& \Me_3 &+& \Me_4, \\
\Mkk_2 & -3\Me_0 &+& \Me_1 &+& \Me_2&,~ &       & &        \\
\end{array}
\quad
\begin{array}{r@{\,=\,}r@{\,}r@{\,}r@{\,}r@{\,}r@{\,}r@{\,}r@{\,}r@{\,}r@{\,}}
\Mh_0 & 5(\Ml_0+\Ml_1) &-& 2\Mp_1 &-&  2\Mp_2, \\
\Mh_1 & 3(\Ml_0+\Ml_1) &-&  \Mp_1 &-&   \Mp_2, \\
\Mh_2 & \Ml_0+\Ml_1,\,    & &        & &         \\
\Mk_0 & -2(\Ml_0+\Ml_1) &+& \Mp_1 &+& \Mp_2, \\
\Mk_1 & -2(\Ml_0+\Ml_1) &+& \Mp_1 &+& \Mp_2, \\
\Mk_2 & -2(\Ml_0+\Ml_1)&.~~ &        & &         \\
\end{array}
\]
The basis of $N(Z_0)$ is of type 1 \Mst 
$\sigma_*(\Me_0)=\Me_0$, $\sigma_*(\Me_1)=\Me_{2}$ and $\sigma_*(\Me_3)=\Me_{4}$.
The basis of $N(Y_0)$ is of type 2 \Mst 
$\sigma_*(\Ml_0)=\Ml_1$ and $\sigma_*(\Mp_1)=\Mp_2$.
The map $\gamma_i$ is the contraction of an exceptional curve $C\subset Y_i$ 
\Mst $[C]=\Me_0-\Me_1-\Me_2$ for $0\leq i\leq 2$. 
For computing the psuedo adjoint chain 
we notice that $\MmfM(\Mhh_i+\Mkk_i)=\Mhh_{i+1}$ and $\MmfF(\Mhh_i+\Mkk_i)=\Me_0-\Me_1-\Me_2$
for $0\leq i\leq 1$.
We verify that $\gamma_i^*\Mh_i=\Mhh_i$ 
and $\Mkk_i-\gamma_i^*\Mk_i=\Me_0-\Me_1-\Me_2$
for $0\leq i\leq 2$,
in accordance with \LEM{par}.\Mrefmclaim{a} and \EQN{k} \Mresp.
We know from case (4) at \PRP{mp} and from \PRP{pp} 
that $S(Y_2,\Mh_2)=\{~\Ml_0+\Ml_1~\}$ and $S(Y_0,\Mh_0)=S(Y_1,\Mh_1)=\{~\Ml_0+\Ml_1-\Mp_1-\Mp_2~\}$.
We convert from a type 2 to a type 1 basis via the map
\[
\Mfun{\gamma_0^*}{N(Y_0)}{N(Z_0)}{(~\Ml_0,~\Ml_1,~\Mp_1,~\Mp_2~)}{(~\Me_0-\Me_1,~\Me_0-\Me_2,~\Me_3,~\Me_4~)}.                             
\]
By \LEM{par}.\Mrefmclaim{b} we have that
$S_*(Z_2,\Mhh_2)=\{~ 2\Me_0-\Me_1-\Me_2 ~\}$
and 
$S_*(Z_0,\Mhh_0)=S_*(Z_1,\Mhh_1)=\{~ 2\Me_0-\Me_1-\Me_2-\Me_3-\Me_4 ~\}$.
Notice that $(Z_2,\Mhh_2)$ is the ruled pair of a sphere in 3-space, 
as discussed in \EXM{sphere}.
\Mend
\end{example}

\newpage
\begin{algorithm}
\textbf{(minimal families)}  
\label{alg:fam} 
\begin{itemize}[itemsep=-3pt,topsep=0pt]
\item \textbf{Input:}
A birational map $\Mdashrow{\McalH}{\MbbP^2}{X}$.

\item \textbf{Output:}
The set $\Gamma\subset N(Z_0)$ of classes of minimal families,
where $\Mrow{\tau_1}{Z_0}{\MbbP^2}$ be the resolution of the base locus of $\McalH$
as in \EQN{res}.

\item \textbf{Method:}
\begin{enumerate}[topsep=0pt, leftmargin=0mm]
\item
Compute $\Mhh_0,\Mkk_0\in N(Z_0)$ using a basepoint analysis algorithm
(see \citep[Algorithm~1 and Section~4.3]{nls-bp}).

\item
$i:=0$;
\textbf{while} $h^0(\Mhh_i+\Mkk_i)>1$ \textbf{do}
\begin{enumerate}[itemsep=-3pt,leftmargin=8mm, label=(\roman*)]

\item
Compute linear series $Q=(q_j)_j$ of curves in $\MbbP^2$ 
\Mst $\tau_1^*Q$ has class $\Mhh_i+\Mkk_i$ in $N(Z_i)$ \citep[Algorithm~2]{nls-bp}. 
Set $Q':=\left(\frac{q_j}{g}\right)_j$ where $g$ is the polynomial gcd of 
the generators $(q_j)_j$.
Verify which of the basepoints computed at step (1) 
are basepoints of $Q'$ and determine their multiplicities.

\item
{\it Let the tuple $(p_t)_{t\in T_{i+1}}$ correspond to the basepoints of $Q'$ 
with multiplicities $(m_t)_{t\in T_{i+1}}$
and let $m_0$ denote the degree of the curves in $Q'$.
Let $\Me_t$ denote the class of the pullback of the exceptional curve
that contracts via $\tau_1$ to $p_t$.}
\\
$N(Z_{i+1}):=\langle~\Me_t~|~t\in \{0\}\cup T_{i+1} ~\rangle$;
\\
$\Mhh_{i+1}:=m_0\Me_0-\Msum{t\in T_{i+1}}{}m_t\Me_t$;\qquad
$\Mkk_{i+1}:=-3\Me_0+\Msum{t\in T_{i+1}}{}\Me_t$; 

\item 
$i:=i+1$;
\end{enumerate}

%
%
%

\item
$\widehat{\Psi}:=\Psi_0\cup\Psi_1\cup\Psi_2\cup\Psi_4\cup\left\{~2\Me_0-\Me_1-\Me_2, ~\frac{-2}{\Mhh_i\Mkk_i}\Mhh_i,~\frac{-2}{(2\Mhh_i+\Mkk_i)\Mk_i}(2\Mhh_i+\Mkk_i) ~\right\}$;
\\
\textbf{while} $i\geq 0$ \textbf{do}
\begin{enumerate}[leftmargin=8mm, label=(\roman*)]
\item $\Gamma:=\Mset{ f\in\widehat{\Psi} }{ \Mhh_i\cdot f=\text{min}\Mset{\Mhh_i\cdot g}{g\in \widehat{\Psi}} }$;
\\
$\Gamma:=\Mset{ f\in \Gamma }{\text{ generators of linear series with class }f\text{ are coprime}}$;
\item {\bf if} $\Mset{\Mkk_i\cdot f}{f\in\Gamma}=\{-2\}$ {\bf then return} $\Gamma$;
\item $i:=i-1$; 
\end{enumerate}
\item {\bf return} $\Gamma$;
\Mend
\end{enumerate}

\end{itemize}
\end{algorithm}

\newpage
\begin{theorem} 
The output specification of \ALG{fam} is correct.
\end{theorem}

\begin{proof}
At step (1) we compute the generators of the NS-lattice of the pair $(Z_0,\Mhh_0)$.
At step (2) we compute 
the generators of $N(Z_i)$ together with $\Mhh_i$ and $\Mkk_i$,
for each pair $(Z_i,\Mhh_i)$ in the pseudo adjoint chain  of $(Z_0,\Mhh_0)$.
Step (3) is a direct consequence of \LEM{par}
and \THM{cls}: we go backwards through the pseudo adjoint chain,
computed at step (2),
until all minimal families 
have canonical degree $-2$. Note that by \LEM{par}.\Mrefmclaim{a} the
inclusion $\Mhookrow{\iota}{N(Z_{i+1})}{N(Z_{i})}$ preserves the generators 
of a type 1 basis $\Mmod{\Me_0,\ldots,\Me_r}$.
We remark that computing the generators of
the linear series at step 3.(i) is done as at step 2.(i) using \citep[Algorithm~2]{nls-bp}.
\end{proof}

\begin{remark}
\textbf{(representation of complete minimal families)}
\\
Suppose that $[F]$ in output $\Gamma$ of \ALG{fam}, 
is the class of a complete minimal family $F\subset X\times \MbbP^1$.
We explain in \citep[Section~4]{nls-algo-fam},
how to compute a reparametrization 
$\Mdashrow{\McalP}{\MbbP^1\times\MbbP^1}{X}$ of $\McalH$ \Mst 
for arbitrary but fixed $b\in\MbbP^1$, the map $\McalP(a,b)$ parametrizes member $F_b\subset X$ with parameter $a\in\MbbP^1$
(see \SEC{fam} for the definition of $F_b$).
\Mend
\end{remark}

\begin{remark}
\textbf{(representation of non-complete minimal families)}
\\
Suppose that $[F]$ in output $\Gamma$ of \ALG{fam}, is the class of a 
non-complete minimal family $F\subset X\times A$. 
We sketch an algorithm for computing equations for $F$. 
Recall that a linear series on $X$
is represented as a linear series on $\MbbP^2$, such that 
a curves in the respective linear series are related via diagram 
\EQN{res} (see \citep[Section~4.3]{nls-bp} and \citep[Section~4]{nls-algo-fam} for more details).

After step (2) in \ALG{fam} we computed $\Mhh_\ell$. 
We assume for simplicity that  
$\Mrow{\gamma_\ell}{Z_\ell}{Y_\ell}$ in \LEM{par} is an isomorphism so that $\Mhh_\ell=\Mh_\ell=-\Mk_\ell$.
We know from \PRP{mp} that either 
$[F]\in\Psi_2$ with $\Mhh_\ell^2=2$, 
$[F]\in\Psi_2$ with $\Mhh_\ell^2=1$
or 
$[F]\in\Psi_4$ with $\Mhh_\ell^2=1$.
We make a case distinction.

Suppose that $[F]\in\Psi_2$ and $\Mhh_\ell^2=2$. 
By \LEM{psi2a},
the linear series with class $[F]=\Mhh_\ell$ defines a map $\Mrow{\alpha}{\MbbP^2}{\MbbP^2}$
whose branching locus $B\subset\MbbP^2$ is a quartic curve.
We compute its dual $B^*\subset \MbbP^{2*}$ so that 
a point in $B^*$ correspond to a tangent line of $B$. 
Thus each irreducible component of $B^*$ defines a family of tangent lines of $B$.
Notice that such family is represented as a linear equation with indeterminate coefficients.
For each coefficient that satisfies the algebraic equations of the component, the corresponding line is tangent to $B$. 
Now suppose that $A$ is the smooth model of some component of $B^*$.
In this case, the image of the corresponding tangent lines with respect to $\McalH\circ\alpha^{-1}$ defines $F\subset X\times A$.
We use either Gr\"obner basis or resultants for inverting rational maps and
computing the preimage of varieties under rational maps \cite{sch11}. 

Suppose that $[F]\in\Psi_2$ with $\Mhh_\ell^2=1$. 
By \LEM{psi2b},
the linear series of $[F]=\Mhh_\ell+e$ defines a map $\Mrow{\beta}{\MbbP^2}{\MbbP^2}$,
where $e$ is the class of an exceptional curve.
We proceed as in the previous paragraph, but with $\beta$ instead of $\alpha$.

Suppose that $[F]\in\Psi_4$ with $\Mhh_\ell^2=1$. 
By \LEM{psi4}, the linear series of $[F]=2\Mhh_\ell$ defines a map $\Mrow{\gamma}{\MbbP^2}{Q}$,
whose branching locus $B\subset Q$ is a sextic curve on a quadric cone $Q\subset\MbbP^3$.
The dual $B^*\subset \MbbP^{3*}$ is a curve whose points correspond to osculating planes 
of $B\subset\MbbP^3$. Points on a line in the tangent developable $T^*\subset \MbbP^{3*}$ of $B^*$
correspond to planes though a tangent line of $B$. Thus a family of bitangent
planes of $B$ correspond to a component of the singular locus of $T^*$. 
The family of bitangent planes of $B$ that are also tangent to $Q$ define 
a triple conic component of the singular locus of $T^*$ and the lines in $Q$
are 2:1 coverings of elliptic curves in $X$ \citep[Section 8.8.3]{dol1}.
The curve $B^*$ also defines a component in the singular locus of $T^*$, 
which do not correspond to bitangent planes but to osculating planes. 
Suppose that $S^*\subset T^*$ is a component of the singular locus that is not 
a triple conic or $B^*$. 
Each point $P^*\in S^*$ defines a conic in $C_{P}:=Q\cap P$ where $P\subset\MbbP^3$ is the bitangent plane
corresponding to $P^*$. 
Thus we obtain a family of bitangent plane sections $(C_P)_{P^*\in S^*}$.
If $A$ is the smooth model of $S^*$, then
the pullback of $C_P\subset Q$
via $\McalH\circ \gamma^{-1}$ defines a curve in $F\subset X\times A$.
See \citep[Example~38]{nls-f2} for a worked out example with equations.
\Mend
\end{remark}

\section{Acknowledgements}

I thank J. Schicho for his support and interesting discussions.
I thank M. Skopenkov for interesting comments.

\bibliography{geometry}

\paragraph{address of author:}
Johann Radon Institute for Computational and Applied 
Mathematics (RICAM), Austrian Academy of Sciences
\\
\textbf{email:} niels.lubbes@gmail.com

\end{document}